\documentclass[12pt]{amsart}

\usepackage{amsmath}
\usepackage[psamsfonts]{amssymb}
\usepackage{amsfonts}
\usepackage{algorithm2e}
\usepackage{graphicx}
\usepackage{verbatim}
\usepackage{dsfont}
\usepackage{tikz-cd}
\usepackage{mathrsfs}
\usepackage{color}
\usepackage[english]{babel}
\usepackage{fontenc}
\usepackage{indentfirst}
\usepackage{tikz}
\usetikzlibrary{matrix,arrows,decorations.pathmorphing}
\usepackage{algorithm2e}
\usepackage[all]{xy}
\usepackage[T1]{fontenc}
\usepackage{hyperref}
\usepackage{mathtools}
\usepackage{multicol}

\newtheorem{theorem}{Theorem}[section]
\newtheorem{prop}[theorem]{Proposition}
\newtheorem{lemma}[theorem]{Lemma}
\newtheorem{cor}[theorem]{Corollary}

\newtheorem{ex}[theorem]{Example}

\newtheorem{remark}[theorem]{Remark}
\newtheorem{claim}[theorem]{Claim}

\def\ep{\epsilon}

\def\R{\mathbb{R}}
\def\Z{\mathbb{Z}}
\def\N{\mathbb{N}}

\def\C{\mathbb{C}}

\def\T{\mathbb T}
\def\Q{\mathbb Q}
\def\cal R{\mathcal R}

\DeclarePairedDelimiter\ceil{\lceil}{\rceil}
\DeclarePairedDelimiter\floor{\lfloor}{\rfloor}

\begin{document}

\title{The shape invariant of symplectic ellipsoids}

\author{Richard Hind}
\email{hind.1@nd.edu}
\address{Department of Mathematics, University of Notre Dame, Notre Dame, IN 46556, USA}

\author{Jun Zhang}
\email{jun.zhang.3@umontreal.ca}
\address{Department of Mathematics and Statistics, University of Montreal, C.P. 6128 Succ. Centre-Ville Montreal, QC H3C 3J7, Canada}

\maketitle 

\begin{abstract}
The shape invariant of a symplectic manifold encodes the possible area classes of embedded Lagrangian tori. Potentially this is a powerful invariant, but for most manifolds the shape is unknown. We compute the shape for 4 dimensional symplectic ellipsoids where the ratio of the factors is an integer. The full shape invariant gives stronger embedding obstructions than results by considering only monotone tori.\end{abstract}

\section{Introduction}

Studying existence of Lagrangian submanifolds is a fundamental problem in symplectic topology. Lagrangian tori in the trivial homology class are always present, but not with arbitrary area class. Even when our symplectic manifold $X$ is an elementary bounded toric domain, little is known on existence beyond the torus fibers. One advance is due to \cite{HO19}, which deals with four-dimensional ball and polydisks. In this paper, we will give a classification of Lagrangian tori for a large family of four-dimensional ellipsoids, in other words we compute a reduced version of their {\it shape invariants} (see (\ref{dfn-red-si}) below).

To be precise, let us fix some notation. Let $\R^4$ be the vector space equipped with the standard symplectic form $\omega_{\rm std} = dx_1 \wedge dy_1 + dx_2 \wedge dy_2$. Identify $\R^4$ with $\C^2$ where $z_i = x_i + \sqrt{-1} y_i$ for $i = 1, 2$. The four-dimensional ellipsoid $E(a,b)$ with $a \leq b$ is defined by 
\[ E(a,b) = \left\{(z_1, z_2) \in \C^2 \,\bigg| \, \frac{\pi |z_1|^2}{a} + \frac{\pi |z_2|^2}{b} < 1\right\}. \]
When $a=b$, $E(a,b) = B(a)$, the four-dimensional ball. Let $L \subset E(a,b)$ be an embedded Lagrangian torus. Recall there are  two  cohomology classes $\Omega \in H^1(L, \R)$ and $\mu \in H^1(L,\Z)$ associated to this $L \subset \R^4$. The first is the area class and is defined by $\Omega = [\lambda_{\rm std}]$ where $\lambda_{\rm std} = y_1 dx_1 + y_2 dx_2$ is the standard Liouville form of $\R^4$, that is, a primitive of $\omega_{\rm std}$. Equivalently, for $e \in H_1(L, \Z)$ we have $\Omega(e) = \int _D u^* \omega$ where $D$ is a disk and $u:(D,\partial D) \to (\C^2,L)$ verifies $u_*[\partial D]=e$. The second class is the Maslov class. If $u:S^1 \to L$ with $u_*[S^1]=e$ then $\mu(e)$ is the Maslov class of the loop of Lagrangian subspaces $T_{u(t)}L \subset \C^2$ (see \cite{RS93}). Since $L$, as a torus, is orientable, the Maslov class is always even. A nice package that encodes the values of $\Omega$ for embedded Lagrangian tori $L \subset E(a,b)$ is the shape invariant of $E(a,b)$. In general, the shape invariant of a domain $X$, denoted by ${\rm Sh}(X)$, was defined and considered  in \cite{Sik89}, \cite{Sik91}, and \cite{Eli91}. It provides a powerful non-linear symplectic invariant to obstruct embeddings between domains. Explicitly, 
\[ {\rm Sh}(X) := \left\{\frak{s} \in H^1(\mathbb T^2, \R)\,| \, \frak{s} = f^* \Omega\,\,\mbox{for a Lagrangian embedding $f:\mathbb T^2 \hookrightarrow X$} \right\}.\]
This definition is slightly unwieldy since Lagrangian embeddings can be composed with diffeomorphisms of $\mathbb T^2$, so we get many classes for each embedded torus. To simplify the discussion, we make the following useful observation. Recall that a Lagrangian torus is called monotone if there exists a constant $c$ such that $\Omega = c \mu$. 

\begin{lemma} \label{lemma-1} If $L \subset \R^4$ is a Lagrangian torus, then there exists an ordered integral basis \footnote{Here an integral basis means every other (integral) class can be written as an integer combination. Equivalently, we are only considering the base changes induced by the diffeomorphisms of $\mathbb T^2$.} $(e_1, e_2)$ of the homology group $H_1(L, \Z)$ such that $\mu(e_1) = \mu(e_2) =2$ and $0< \Omega(e_1) \le \Omega(e_2)$. Moreover, in the non-monotone case when $\Omega(e_1) \neq \Omega(e_2)$, there exists a unique such basis with $0< 2\Omega(e_1) \leq \Omega(e_2)$. \end{lemma}

The proof of Lemma \ref{lemma-1} is given in Section \ref{sec-red}, and it is based on a fundamental result, originally appearing as Theorem A in \cite{Vit90} or Theorem 2.1 in \cite{Pol91} that, in our set-up, there always exists a class $e \in H_1(L; \R)$ with $\mu(e)=2$ and $\Omega(e)>0$. Then in this paper we will mainly consider the following reduced version of ${\rm Sh}(X)$, 
\begin{align} \label{dfn-red-si} 
{\rm Sh}^+(X) & := \left\{(\Omega(e_1), \Omega(e_2 ))\in \R^2 \,\bigg| \, \begin{array}{l} \mbox{$L$ is Lagrangian torus in $X$ and } \\ \mbox{$(e_1, e_2)$ is given by Lemma \ref{lemma-1}.} \end{array} \right\}\\
& = \left\{\frak{s}(f^{-1}_*(e_1, e_2)) \in \R^2 \,\bigg| \, \begin{array}{l} \mbox{$L = f(\mathbb T^2)$ for a Lagrangian emb. $f$,} \\ \frak{s} \in {\rm Sh}(X),\,\,\mbox{and}\,\, \mbox{$(e_1, e_2)$ in Lemma \ref{lemma-1}} \end{array} \right\}. \nonumber
\end{align} 
\begin{remark} When $L$ is non-monotone, the basis $(e_1, e_2)$ is uniquely determined by the second conclusion of Lemma \ref{lemma-1}, thus there is no ambiguity in the definition (\ref{dfn-red-si}). When $L$ is monotone, though the basis $(e_1, e_2)$ is not uniquely determined, one readily verify that if $w = \Omega(e_1) = \Omega(e_2)$ for one such basis, then any other Maslov 2 basis $(e'_1, e'_2)$ also satisfies $w = \Omega(e'_1) = \Omega(e'_2)$. Therefore, (\ref{dfn-red-si}) is well-defined. Moreover, if we identify $H^1(T^2, \R)$ with  $\R^2$ by evaluating on a fixed basis then ${\rm Sh}^+(X) \subset {\rm Sh}(X)$, and if $\R^2$ is given by coordinates $(w_1, w_2)$, then ${\rm Sh}^+(X)$ lies in the region $\{w_1>0, w_2 \geq 2 w_1\} \cup\{w_1>0, w_1 = w_2\}$. 
The reduced shape invariant ${\rm Sh}^+(X)$ serves as a fundamental domain of ${\rm Sh}(X)$ under the action of diffeomorphisms of $\mathbb T^2$. \end{remark}

\subsection{Main results} The main result in this paper is the following.

\begin{theorem} \label{thm-1} For ellipsoid $E(a,b)$ with $\frac{b}{a} \in \N_{\geq 2}$, 
\[ {\rm Sh}^+(E(a,b)) = \left\{(w_1, w_2) \in \R_{>0}^2 \, \bigg| \, \begin{array}{c}  w_1 < \frac{a}{2}, \,\, w_2 >0 \\ \mbox{\rm or} \\ w_1 \geq \frac{a}{2},\,\,\frac{w_1}{a} + \frac{w_2}{b} <1 \end{array} \right\} \cap \left( \begin{array}{c} \{w_2 \geq 2w_1\} \\ \cup \\ \{w_1 =w_2\}\end{array} \right). \]
In other words, the reduced shape invariant of ellipsoid $E(a,b)$, when restricted to the region $\{w_1>0, w_2 \geq 2 w_1\} \cup\{w_1>0, w_1 = w_2\}$, is the union of the moment triangle of $E(a,b)$ with a strip with width $\frac{a}{2}$.  
\end{theorem}

For a pictorial illustration of Theorem \ref{thm-1}, see Example \ref{ex-hsi} in Section \ref{sec-app}. Note that when $\frac{b}{a} =1$, that is, $E(a,b) = B(a)$, a similar result is proved in Theorem 2 of \cite{HO19}, where the essential difference is that $w_1 < \frac{a}{3}$ instead of $w_1 < \frac{a}{2}$. 

We will denote by $L(1,x) \hookrightarrow E(a,b)$ a Lagrangian embedding of the torus $\mathbb T^2$ into $E(a,b)$ such that the standard basis is mapped to a basis $(e_1, e_2)$ given by Lemma \ref{lemma-1} and satisfying $\Omega(e_1)=1$ and $\Omega(e_2) = x$. Hence, $L(1,x) \hookrightarrow E(a,b)$ with either $x=1$ or $x \ge 2$ if and only if $(1,x) \in  {\rm Sh}^+(E(a,b))$. 

\begin{theorem} \label{thm-2} Suppose $E(a,b)$ satisfies $\frac{b}{a} \in \N_{\geq 2}$. Then we have the following conclusions. 
\begin{enumerate}
\item $L(1,1) \hookrightarrow E(a,b)$ if and only if $1 < b(1 - \frac{1}{a})$.
\item If $x \ge 2$, then $L(1,x) \hookrightarrow E(a,b)$ if and only if either $a > 2$ or $x < b(1-\frac{1}{a})$.
\end{enumerate}
\end{theorem}

In fact, Theorem \ref{thm-2} directly implies Theorem \ref{thm-1}.

\begin{proof} [Proof of Theorem \ref{thm-1}, assuming Theorem \ref{thm-2}] The hypothenuse of the moment triangle of $E(a,b)$ is $w_2 = - \frac{b}{a} w_1 + b$. Note that $L(w_1, w_2) \hookrightarrow E(a,b)$ if and only if $L(1, \frac{w_2}{w_1}) \hookrightarrow E(\frac{a}{w_2}, \frac{b}{w_2})$. Then by Theorem \ref{thm-2}, there are two cases. 
\begin{itemize}
\item[(1)] If $\frac{w_2}{w_1}= 1$, that is, $w_1 = w_2$, then we have $1< \frac{b}{w_2}(1 - \frac{w_1}{a})$, which is $w_2 < - \frac{b}{a} w_1 + b$. 
\item[(2)] If $\frac{w_2}{w_1} \geq 2$, that is, $w_2 \geq 2 w_1$, we have either $\frac{a}{w_1} \geq 2$, which is $w_1 < \frac{a}{2}$, or 
$w_2 < - \frac{b}{a} w_1 + b$ as in case (1). 
\end{itemize}
Thus we get the desired conclusion. \end{proof}

In what follows, we will be mainly interested in Theorem \ref{thm-2}. The proof follows a similar idea as the proof of Theorem 2 in \cite{HO19}. There are two major differences, which increases the difficulty in the situation of a general $E(a,b)$. First, instead of compactifying $B(a)$ into a projective plane in $\C P^2$, we will directly carry out an SFT-type argument with positive asymptotic ends on the contact manifold $\partial E(a,b)$. Second, to initiate the neck-stretching process, we build up by hand a finite energy curve in the symplectic cobordism between $E(a,b)$ and the thin ellipsoid $E(\ep,\ep S)$ (with a sufficiently large $S$ and a sufficiently small $\ep$), while \cite{HO19} was able to take advantage of the existence of an analogous curve (in fact a plane), already known from \cite{HK14}, \cite{HK18}. 

\begin{remark} [Communicated by M. Hutchings] \label{rmk-hutchings} A much more general conclusion like (1) in Theorem 1.4 above is in some sense known. By combining work (in progress) of Miguel Pereira using similar techniques to Cieliebak-Mohnke's paper \cite{CM18} and Lemma 1.19 in \cite{GH18} involving the Gutt-Hutchings capacities, one can show that $L(1, …, 1) \hookrightarrow E(a_1, …, a_n)$ if and only if $\frac{1}{a_1} + \cdots \frac{1}{a_n} <1$. Here, the symplectic ellipsoid $E(a_1, …, a_n)$ is of any dimension $2n \geq 4$ and without any integrality assumption on factors. Since (1) in Theorem 1.4 is a byproduct of the proof of (2) in Theorem 1.4, we include it here for the sake of completeness. \end{remark}

\noindent {\bf Applications.} The reduced shape invariant ${\rm Sh}^+(X)$ can provide a useful tool to obstruct symplectic embeddings between star-shaped domains of $\C^2$. Here are some direct consequences from Theorem 3 in \cite{HO19} and Theorem \ref{thm-1}. Recall that a symplectic polydisk $P(a,b)$ with $a \leq b$ is defined by $P(a,b) := \{(z_1,z_2) \in \C^2 \,|\, \pi|z_1|^2<a, \,\,\pi|z_2|^2<b\}$. The following result is proved in Section \ref{sec-app}.

\begin{theorem} \label{thm-3} We have the following obstructions of symplectic embeddings. 
\begin{itemize}
\item[(1)] Let $P(a,b)$ and $P(c,d)$ be polydisks with $a \leq b$ and $c \leq d$. If $b>d$, then 
\[ \mbox{$P(a,b) \hookrightarrow P(c,d)$} \,\, \mbox{implies that} \,\,\frac{c}{a} \geq 2.\] 
\item[(2)] Let $P(1,a)$ be a polydisk with $a \geq 2$ and $E(c,bc)$ be an ellipsoid satisfying $b \in \N_{\geq 2}$ and $1\leq c\leq 2$. Then 
\[ \mbox{$P(1,a) \hookrightarrow E(c,bc)$} \,\, \mbox{if and only if}\,\, a + b\leq bc.\] 
\end{itemize}
\end{theorem}

Here are a few remarks that relate Theorem \ref{thm-3} to other results in the literature. 

\medskip

\noindent (i) For (1) note that the Gromov width shows a necessary condition for an embedding $P(a,b) \hookrightarrow P(c,d)$ is that $a \leq c$. If also $b \le d$ then we have an inclusion, so the interesting case is precisely $b>d$. The conclusion (1) in Theorem \ref{thm-3} solves conjecture in Remark 1.8 in \cite{Hut16}. 
Given work of Choi \cite{Choi16} (which proves Conjecture A.3 in \cite{Hut16}) this also follows from the methods in \cite{Hut16}. 

\medskip

\noindent (ii) An upper bound of $c$ is certainly necessary to obtain a sharp obstruction as in (2) in Theorem \ref{thm-3}. In fact, by symplectic folding, there exists a symplectic embedding $P(1,10) \hookrightarrow P(2, 6)$, and by inclusion $P(2, 6) \subset E(5,10) = E(c, bc)$ with $c=5$ and $b=2$. However, $a = 10> 8 = bc-b$.

On the other hand the assumption that $a \geq 2$ in (2) is not necessary. The case when $1 \leq a \leq 2$ is proved by Theorem 1.5 in \cite{Hut16}. It also follows from our Theorem \ref{ham=red} on the Hamiltonian shape invariant. (It is interesting that our methods require more delicate techniques to deal with the cases when the results are already known from ECH theory.) Moreover, the recent work \cite{DNNWY20} provides a refinement of Theorem 1.5 in \cite{Hut16} in the sense that it considers the target ellipsoid $E(c,bc)$ where $b$ is a half-integer.  

\medskip

\noindent (iii) Theorem 1.2.7 in \cite{Sie20} proves a more general stablilized version of (1) in Theorem \ref{thm-3} in the case when $c=d$, i.e., the obstruction of the embedding $P(a,b) \hookrightarrow P(c,c)$. 
Finally, the result (1) in Theorem \ref{thm-3} in the case when $b/a \ge 2$ has been announced recently in \cite{Irv19} (again even for the stabilized version). We list the results here to demonstrate how quickly a comparison between the reduced shape invariants implies the obstructions (see Section \ref{sec-app}). 

\subsection{Discussion} We end this introduction with short discussions on two subjects. 
\subsubsection{Relations to classical symplectic capacities} Due to Proposition \ref{prop-ham-inv} and its following paragraph, the reduced shape invariant ${\rm Sh}^+(X)$ behaves very much like a symplectic capacity. However, ${\rm Sh}^+(X)$, taking values in subsets of the plane, may contain more information than other classical symplectic capacities. To obtain an explicit relation, consider the following value 
\[ c^1(X): = \sup\{w>0 \,| \, w_1 = w_2 =w \in {\rm Sh}^+(X)\}.\]
Points on the diagonal of ${\rm Sh}^+(X)$ correspond to the embeddings of monotone Lagrangian tori inside $X$, and at least for the ellipsoids and polydisks considered here this is in fact the Lagrangian capacity $c_L(X)$ defined in \cite{CM18}.  Example \ref{ex-hsi} in Section \ref{sec-app} shows that 
\[ c^1(B(R)) = \frac{R}{2} \,\,\,\,\,\mbox{and}\,\,\,\,\, c^1(Z(R)) = R. \]
In particular, the symplectic capacity $c^1$ is not normalized as shown in Corollary 1.3 in \cite{CM18}. In general, for any real number $\lambda \geq 2$, consider
\begin{equation} \label{lambda-capacity}
c^{\lambda}(X) = \sup\{w_2>0 \, | \, (w_1, w_2) \in {\rm Sh}^+(X) \cap \{w_2 = \lambda w_1\} \}. 
\end{equation} 
These serve as a continuous family of non-normalized symplectic capacities (with roots in the embeddings of {\it non-monotone} Lagrangian tori). In fact, the proof of Theorem \ref{thm-3} can be reformulated in terms of $c^{\lambda}$. It will be an interesting direction to explore more applications of $c^{\lambda}$. The following example shows a preliminary example. 

\begin{ex} In the case when $\frac{b}{a} \in \N_{\ge 2}$ we see that $c^1(E(a,b)) = \frac{ab}{a+b}$ while $c^2(E(a,b)) = \frac{2ab}{2a+b}$. From this it follows that if $c^2(E(a,b)) \le c^2(E(a',b'))$ and $a<a'$ (that is, the Gromov width gives no obstructions), then automatically we have  $c^1(E(a,b)) \le c^1(E(a',b'))$. In other words, the capacity $c^2$ gives strictly stronger restrictions on ellipsoid embeddings than $c^1$.

On the other hand, the full shape invariant for ellipsoids with integer ratios only implies that if $E(a,b) \hookrightarrow E(a', b')$ then $\frac{a'}{a} + \frac{b'}{b} \ge 2$ which follows from the Gromov width and volume obstructions.

\end{ex}

\subsubsection{Hamiltonian shape invariant}\label{hamshape} For domains in Euclidean space one can consider a slightly different version of the shape invariant, which relates to the study in \cite{HO19}. Explicitly, consider those embedded Lagrangian tori $L \subset X$ which can be realized as the image of a Lagrangian {\it product} torus under Hamiltonian diffeomorphisms of $\C^2$. These Lagrangian tori are called Hamiltonian tori. Recall that the Lagrangian product tori in $\C^2$, denoted by $L_H(w_1,w_2)$, are defined by 
\[ L_H(w_1,w_2) := \left\{(z_1, z_2) \in \C^2 \, \big| \, \pi|z_1|^2 = w_1\,\, \pi|z_2|^2 = w_2\right\}. \]
If $L$ is a Hamiltonian torus, that is, $L = \Phi(L_H(w_1,w_2))$ for some Hamiltonian diffeomorphism $\Phi$, then with respect to the natural basis  its area class  is $(w_1,w_2)$. For brevity, we write $L_H(w_1, w_2) \hookrightarrow X$ if there exists a Hamiltonian diffeomorphism $\Phi$ on $\C^2$ such that $\Phi(L_H(w_1,w_2)) \subset X$. For a domain $X \subset \C^2$, the following definition gives another version of the shape invariant.  
\begin{align}\label{dfn-hsi} 
{\rm Sh}_{H}(X) & : = \left\{(w_1, w_2) \in \R_{>0}^2 \, \big| \, L_H(w_1, w_2) \hookrightarrow X \right\}. 
\end{align}
This ${\rm Sh}_{H}(X)$ is called the {\it Hamiltonian shape invariant of $X$}. Observe that the requirement of the product tori in the definition of ${\rm Sh}_{H}(X)$ means it is no longer invariant under the action of the diffeomorphisms of $\mathbb T^2$. As a matter of fact, ${\rm Sh}_H(X)$ is only symmetric with respect to the line $w_1= w_2$. For simplicity then, we consider a corresponding reduced version, that is, ${\rm Sh}_{H}^+(X) := {\rm Sh}_{H}(X)|_{\{w_1 \leq w_2\}}$. 



We emphasize that when restricted to the fundamental domain $\{w_1>0, w_2 \geq 2 w_1\} \cup\{w_1>0, w_1 = w_2\}$, we have 
\begin{equation}\label{ham=red}
{\rm Sh}_H^+(X)|_{\{w_1>0, w_2 \geq 2 w_1\} \cup\{w_1>0, w_1 = w_2\}} = {\rm Sh}^+(X)|_{\{w_1>0, w_2 \geq 2 w_1\} \cup\{w_1>0, w_1 = w_2\}}.
\end{equation}
Indeed, by definition ${\rm Sh}_H^+(X) \subset {\rm Sh}^+(X)$, and the embedded Lagrangian tori constructed in Section \ref{existence} can all be realized by Hamiltonian diffeomorphisms applied to product tori. However, in general we do not have ${\rm Sh}_H(X) = {\rm Sh}(X)$. For example, the product torus $L_H(1,2)$ lies in $X=E(2,4)$ by inclusion, and so $(1,2) \in {\rm Sh}^+(X)$. By a change of basis this implies that $(3,4) \in {\rm Sh}(X)$, but $(3,4) \notin {\rm Sh}_H(X)$, since, for example, the product torus $L_H(3,4)$ has displacement energy $3$ while $E(2,4)$ has displacement energy $2$ (see Proposition 2.1 in \cite{chesch}). 

\medskip

We are able to determine the reduced Hamiltonian shape invariant of $E(a,b)$, at least if we utilize an unpublished work of K. Siegel \cite{Sieip}. 


\begin{theorem} \label{conj-1} For ellipsoid $E(a,b)$ with $\frac{b}{a} \in \N_{\geq 2}$, 
\[ {\rm Sh}_{H}^+(E(a,b)) = \left\{(w_1, w_2) \in \R_{>0}^2 \, \bigg| \, \begin{array}{c}  w_1 < \frac{a}{2}, \,\, w_2 >0 \\ \mbox{\rm or} \\ w_1 \geq \frac{a}{2},\,\,\frac{w_1}{a} + \frac{w_2}{b} <1 \end{array} \right\} \cap \{w_1 \leq w_2\}, \]
i.e., the Hamiltonian shape invariant of ellipsoid $E(a,b)$, when restricted to the region $\{w_1 \leq w_2\}$, is the union of the moment triangle of $E(a,b)$ with a strip with width $\frac{a}{2}$.  
\end{theorem}

To be precise, our proof of Theorem \ref{conj-1} relies on the nonemptyness of a certain moduli space of holomorphic curves in an ellipsoid cobordism, and for this we require Siegel's work \cite{Sie20,Sieip}. Once this is established, the proof proceeds similarly to that of Theorem \ref{thm-2}. The proof is outlined in section \ref{conjsec}.


\subsection{Outline of the proof of Theorem \ref{thm-2}} \label{ssec-outline} Note that, except the case where $x \geq 2$ and $a>2$, the ``if part'' in both (1) and (2) in Theorem \ref{thm-2} trivially holds by inclusion $L_H(1,x) \subset E(a,b)$. Here, we give the outline of the proof of the obstruction part of Theorem \ref{thm-2}, i.e., the ``only if part'' in Theorem \ref{thm-2}. 

\medskip

Suppose there exists an Lagrangian embedding $L(1,x) \hookrightarrow E(a,b)$. Recall this means the Maslov class of the standard basis is $2$ and the area class evaluates as $(1,x)$. To simplify our discussion, we will always assume that $\frac{b}{a} = k + \delta \notin \Q$ where $\delta>0$ is arbitrarily small. Then there are only two primitive closed Reeb orbits on $\partial E(a,b)$, denoted by $\alpha_1$ with action $a$ and $\alpha_2$ with action $b$. Now, let $V$ be an appropriate Weinstein neighborhood of the image of $L(1,x)$ inside $E(a,b)$ which can be symplectically identified with the unit codisk bundle $U^*_g \T^2$ for a flat metric $g$. In particular, $\partial V$ is a hypersurface of contact type. For any $S$, there exists a sufficiently small $\ep>0$ such that $E(\ep, \ep S)$ is contained in $V$. In other words, we have inclusions 
\[ E(\ep, \ep S) \subset V \subset E(a,b),\]
see Figure \ref{figure-set-up}. 
\begin{figure}[h] 
\includegraphics[scale=0.78]{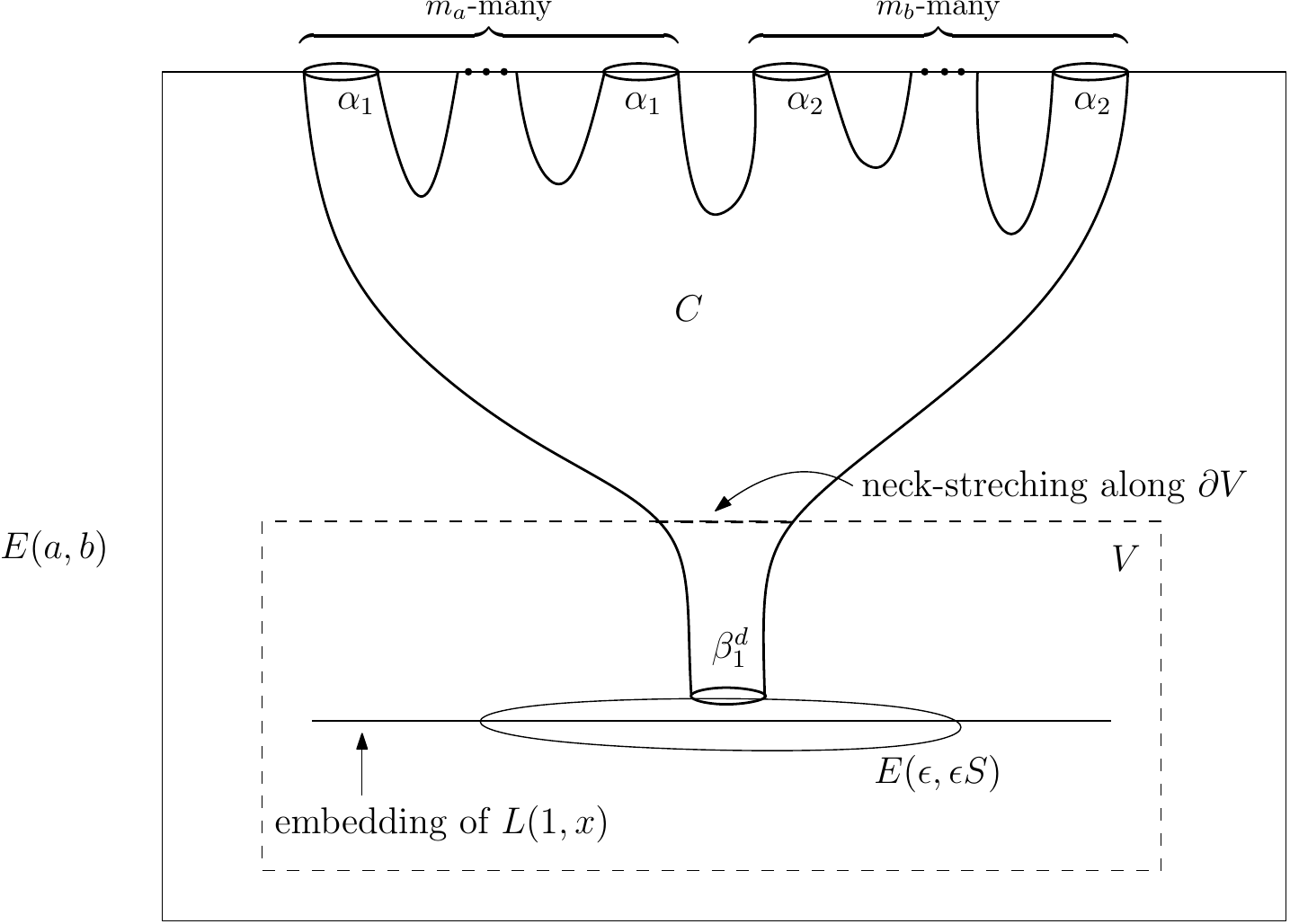}
\caption{Geometric set-up for Theorem \ref{thm-2}.}\label{figure-set-up}
\end{figure}
 In particular, we have a symplectic cobordism $\overline{X} = E(a,b) \backslash E(\ep, \ep S)$ and can choose an almost complex structure with cylindrical ends which is compatible with the contact structures of $\partial E(a,b)$ and $\partial E(\ep, \ep S)$. 
 
 Denote by $\beta_1$ the short primitive Reeb orbit of $E(\ep, \ep S)$, i.e., its period is $\ep$. Consider rigid curves $C$ (that is, its Fredholm index $=0$ and genus $=0$) in $\overline{X}$ with positive ends on $m_a$-many simply-covered $\gamma_a$ and $m_b$-many simply-covered $\gamma_b$, and a single negative end $\beta^d_1$ for some winding number $d$. In fact, we will only consider the curves with $(m_a, m_b) = (m_a, 1)$ where $m_a \to \infty$. The existence and SFT-compactness of these rigid curves are guaranteed by various results in Section \ref{sec-compact-glue}. Due to the general process of neck-stretching along $\partial V$, a sequence of rigid curves $C_n$ with respect to a sequence of almost complex structures $J_n$ which stretch along $\partial V$ converges to a limit holomorphic building $C_{\rm lim}$ by \cite{BEHWZ03}. This limit holomorphic building $C_{\rm lim}$ consists of curves in five layers (from top to bottom), 
 \begin{equation} \label{levels}
 S\partial E(a,b), \,\,\, E(a,b) \backslash V, \,\,\, S\partial V, \,\,\, V \backslash E(\ep, \ep S), \,\,\, \mbox{and}\,\,\, S\partial E(\ep, \ep S). 
 \end{equation}
To simplify the analysis, following the notation (II) in Section 2 of \cite{HO19}, there is one component denoted by $F_0$ with a single negative end $\beta_1^d$ and $T$-many positive ends on $\partial V$ denoted by $\{\gamma_i\}_{i=1}^T$. This is formed by (abstractly) gluing all the curves via their matching ends in the layers up to and including $S \partial V$. Again, following the notation (I) in Section 2 in \cite{HO19}, for each $i \in \{1, …, T\}$, the $i$-th positive end $\gamma_i$ of $F_0$ on $\partial V$ is matched with a component denoted by $F_i$ with only one negative end on $\gamma_i$. This component is formed by gluing all the curves via their matching ends that eventually connect to $\gamma_i$. For the reader's convenience, Figure \ref{figure-climit} illustrates an example of $C_{\rm lim}$. 
\begin{figure}[h] 
\includegraphics[scale=0.75]{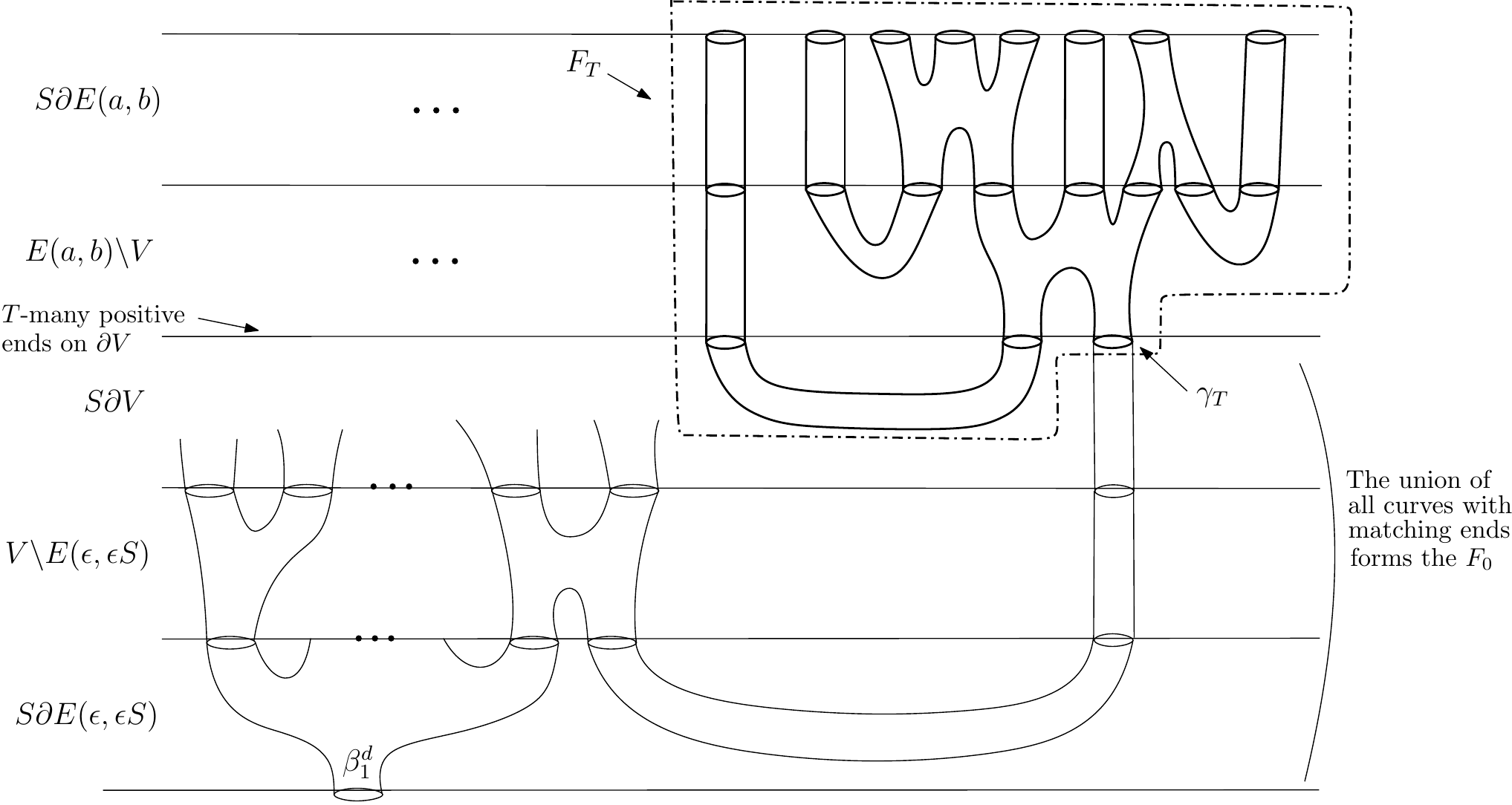}
\caption{An example of $C_{\rm lim}$}\label{figure-climit}
\end{figure}

Here, the component $F_T$ is a union of $11$ curves in layers $S\partial E(a,b), E(a,b) \backslash V$ and $S\partial V$. It ends up with one negative end $\gamma_T$, matching with $F_0$ on $\partial V$, and $8$ positive ends on the top boundary $\partial E(a,b)$. Due to a sophisticated analysis of the configurations of $C_{\rm lim}$ based on actions and Fredholm indices (see subsection \ref{ssec-obs}), the possible configurations of $\{F_i\}_{1 \leq i \leq T}$ are surprisingly restrictive. Within these limited possibilities, the consideration of actions implies the obstructions in Theorem \ref{thm-2}.

\smallskip

\section*{acknowledgements}
This work was completed when the second author holds the CRM-ISM Postdoctoral Research Fellow at CRM, University of Montreal, and the second author thanks this institute for its warm hospitality. Some conversations occurred during a visit of the second author to the University of Notre Dame in Fall 2019; he is indebted to its hospitality. Both authors thank L.~Polterovich and K.~Siegel for feedback on a draft of the paper. Moreover, both authors are grateful to the comments from M.~Hutchings that are related to Remark \ref{rmk-hutchings} and (1) in Theorem \ref{thm-3}, and to the communications from K.~Siegel that are related to the proof of Theorem \ref{conj-1} in Section \ref{conjsec}.

\section{Reduction on shape invariant} \label{sec-red}

In this section, we will give the proof of Lemma \ref{lemma-1} which reduces the consideration from ${\rm Sh}(X)$ to ${\rm Sh}^+(X)$. The latter can be pictured in a precise way. 

\begin{proof} [Proof of Lemma \ref{lemma-1}] By Theorem A in \cite{Vit90} or Theorem 2.1 in \cite{Pol91}, any Lagrangian torus in $\R^4$ admits a (primitive) class $f_1 \in H_1(L; \Z)$ with $\mu(f_1)=2$ and $\Omega(f_1)>0$. We can extend this to an integral basis with a class $f_2$, and, since $\mu(f_2)$ is automatically even, adding a multiple of $f_1$ to $f_2$ gives us a Maslov $2$ basis. If $\Omega(f_2)$ happens to be negative, then we can replace $f_2$ by $2f_1 - f_2$ and so arrive at a basis with $\Omega(f_1), \Omega(f_2)>0$. Further, without loss of generality, we assume that $0< \Omega(f_1) \le \Omega(f_2)$.

In the non-monotone case, we write $w_1 = \Omega(f_1) < \Omega(f_2)=w_2$. We compute the change in areas if we make an integral change of basis preserving the Maslov class. Such changes of basis which preserve orientation are given by matrices of the form
\[
\begin{pmatrix}
a+1 & a\\
-a & 1-a
\end{pmatrix}
\] 
for $a \in \Z$, and the new area classes are 
\[ ( (a+1)w_1 - aw_2, aw_1 + (1-a)w_2) = (w_1 - a(w_2-w_1), w_2- a(w_2-w_1) ).\]
We note that these basis changes preserve the order of areas, and orientation reversing changes would also allow the order to reverse. Note that the first term is positive exactly when $a < \frac{w_1}{w_2-w_1}$. On the other hand, if $w_2- a(w_2-w_1) \geq 2(w_1 - a(w_2-w_1))$, then 
\[ a (w_2-w_1) \geq 2 w_1 - w_2, \,\,\,\,\mbox{which implies} \,\,\,\, a \geq \frac{w_1}{w_2 - w_1} -1. \]
Observe that there exists a unique integer in the interval $[\frac{w_1}{w_2 - w_1} -1,  \frac{w_1}{w_2-w_1})$, that is, 
\begin{equation} \label{unique}
a= \ceil*{\frac{w_1}{w_2-w_1}} -1.
\end{equation}
Hence, the desired ordered basis is $e_1 = (a+1)f_1 - af_2$ and $e_2 = af_1 + (1-a) f_2$ with $a$ chosen as (\ref{unique}). 
\end{proof}

\section{ECH theory and invariants}

\subsection{ECH background} Embedded contact homology ${\rm ECH}(Y)$ is a powerful tool associating algebraic and numerical invariants to a closed 3-dimensional contact manifold $Y$. It is the homology of a chain complex that is freely generated by orbit sets, i.e., a finite collection of distinct embedded Reeb orbits of $Y$ with multiplicities, say $\alpha_1^{n_1}…\alpha_k^{n_k}$. In this paper, we will apply ECH to obtain desired embedded holomorphic curves in certain moduli spaces. This is motivated by the following fact. If there is a symplectic embedding of Liouville domains $Y_2 \hookrightarrow Y_1$, then via Seiberg-Witten theory, there exists a well-defined map (called {\rm ECH} cobordism map \cite{HT07})
 \begin{equation} \label{dfn-ech-cm}
 \Phi: {\rm ECH}(\partial Y_1) \to {\rm ECH}(\partial Y_2).
 \end{equation}
 Moreover, non-vanishing of $\Phi$, even on the chain complex level, will predict a possibly {\it broken} $J$-holomorphic current $C$ with {\rm ECH} index $=0$ (\cite{HT13}). Here, a current is a finite set $\{(C_i, m_i)\}$ where each $C_i$ is an irreducible somewhere injective $J$-holomorphic curve in either the cobordism level or a symplectization level. Usually, $\mathcal M_{\rm current}(\gamma_1, \gamma_2)$ denotes the moduli space of currents with positive end on the orbit set $\gamma_1$ and negative end on the orbit set $\gamma_2$. For the general background of ECH theory, see \cite{Hut11}, \cite{Hut14} and Section 2 in \cite{C-GH18}. Since in this paper our interest lies in ellipsoids and symplectic cobordisms between them, we will formulate necessary ingredients of ${\rm ECH}$ only in the cases of ellipsoids. 

\subsection{Numerical invariants} Suppose there exists a symplectic embedding $E(c,d) \hookrightarrow E(a,b)$ where both $\frac{d}{c}$ and $\frac{b}{a}$ are irrational. Denote by $\overline{X}$ the completion of the corresponding symplectic cobordism with respect to a compatible $J$. Following the notation in Section 2.5 in \cite{C-GH18}, denote by $\alpha_1, \alpha_2$ the two primitive Reeb orbits of $E(a,b)$ where the period of $\alpha_1$ is less than the period of $\alpha_2$; similarly, $\beta_1$ and $\beta_2$ denote the two primitive Reeb orbits of $E(c,d)$ with increasing periods. Recall that the ${\rm ECH}$ grading of an orbit set $\alpha_1^{m_1}\alpha_2^{m_2}$ is defined by 
\begin{equation} \label{dfn-ech-gr}
\frac{{\rm gr}(\alpha_1^{m_1}\alpha_2^{m_2})}{2} = (m_1+ m_2)+ m_1 m_2 +  \sum_{i=1}^{m_1} \floor*{\frac{ia}{b}} +  \sum_{i=1}^{m_2} \floor*{\frac{i b}{a}}
\end{equation}
and similarly we define ${\rm gr}(\beta_1^{n_1}\beta_2^{n_2})$. If $C \in \mathcal M_{\rm current}(\alpha_1^{m_1} \alpha_2^{m_2}, \beta_1^{n_1}\beta_2^{n_2})$, then its ${\rm ECH}$ index, denoted by $I(C)$, is equal to ${\rm gr}(\alpha_1^{m_1}\alpha_2^{m_2}) - {\rm gr}(\beta_1^{n_1}\beta_2^{n_2})$. Thus it can be computed by the following explicit formula, 
\begin{align} \label{ech-index}
\frac{I(C)}{2} & = (m_1 +m_2) - (n_1+n_2) + (m_1m_2 - n_1n_2)\\
& \,\,\,\,\, + \left(\sum_{i=1}^{m_1} \floor*{\frac{ia}{b}} +  \sum_{i=1}^{m_2} \floor*{\frac{i b}{a}}\right) - \left( \sum_{j=1}^{n_1} \floor*{\frac{jc}{d}} +  \sum_{j=1}^{n_2} \floor*{\frac{jd}{c}}\right). \nonumber
\end{align}
The existence of any desired $J$-holomorphic curve between orbit sets in $\overline{X}$ is initiated by the following result. 

\begin{prop} \label{prop-broken} Let $E(c,d) \hookrightarrow E(a,b)$ be a symplectic embedding where $\frac{d}{c}, \frac{b}{a}$ are irrational, and $\alpha_1^{m_1}\alpha_2^{m_2}$ and $\beta_1^{n_1}\beta_2^{n_2}$ be orbit sets of $\partial E(a,b)$ and $\partial E(c,d)$ respectively with the same {\rm ECH} gradings. Then there exists a possibly broken $J$-holomorphic current $C$ from $\alpha_1^{m_1}\alpha_2^{m_2}$ to $\beta_1^{n_1}\beta_2^{n_2}$ with $I(C) =0$. 
\end{prop}

\begin{proof} Since the cobordism of $E(c,d) \hookrightarrow E(a,b)$ is diffeomorphic to a product, the {\rm ECH} cobordism map $\Phi: {\rm ECH}(\partial E(a,b)) \to {\rm ECH}(\partial E(c,d))$ in (\ref{dfn-ech-cm}) is an isomorphism. Moreover, Fact 2.2 in \cite{C-GH18} implies that any non-empty orbit set $\alpha_1^{m_1}\alpha_2^{m_2}$ represents a nonzero class in ${\rm ECH}(\partial E(a,b))$. Therefore, $\Phi([\alpha_1^{m_1}\alpha_2^{m_2}]) \neq 0$. Meanwhile, there is a unique orbit set in any grading and $\Phi$ preserves the {\rm ECH} grading, so the orbit set in $\Phi([\alpha_1^{m_1} \alpha_2^{m_2}])$, denoted by $\beta_1^{n_1}\beta_2^{n_2}$, has the same grading as $\alpha_1^{m_1} \alpha_2^{m_2}$. Finally, since $\Phi([\alpha_1^{m_1} \alpha_2^{m_2}]) \neq 0$, there exists a possibly broken $J$-holomorphic current $C$ from $\alpha_1^{m_1}\alpha_2^{m_2}$ to $\beta_1^{n_1}\beta_2^{n_2}$. Moreover, $I(C) = 0$ by (\ref{ech-index}).
\end{proof}

\begin{remark} In general, we need extra assumptions to conclude that the $J$-holomorphic current guaranteed by Proposition \ref{prop-broken} has only a single non-trivial level (so a cobordism level) consisting of a single somewhere injective connected component. \end{remark} 

Another numerical invariant of a $J$-holomorphic current is called the $J_0$ index, denoted by $J_0(C)$. It can be computed by the following explicit formula,
\begin{equation} \label{j0-index}
\frac{J_0(C)}{2} = (m_1m_2 - n_1n_2)  + \left(\sum_{i=1}^{m_1-1} \floor*{\frac{ia}{b}} +  \sum_{i=1}^{m_2-1} \floor*{\frac{ib}{a}}\right) - \left( \sum_{j=1}^{n_1-1} \floor*{\frac{jc}{d}} +  \sum_{j=1}^{n_2-1} \floor*{\frac{jd}{c}}\right). 
\end{equation}
When $C$ is somewhere injective, connected, has genus $g(C)$, then 
\begin{equation} \label{j0-inequality}
J_0(C) \geq 2(g(C)-1 + \delta(C)) + \sum_{\gamma} (2n_{\gamma}-1) 
\end{equation}
where the sum is over all embedded Reeb orbits $\gamma$ at which $C$ has ends and $n_{\gamma}$ denotes the total number of ends of $C$ on $\gamma$. Moreover, $\delta(C)$ is an algebraic count of the number of singularities of $C$, which implies that $\delta(C) \geq 0$. In particular, $\delta(C) =0$ if and only if $C$ is embedded. 

\medskip

Now, we apply Proposition \ref{prop-broken} and $J_0$ index to a concrete example. This example will be useful later in the paper. 

\begin{ex} \label{ex-1} Consider a trivial embedding (inclusion) $\lambda E(1, k+1 +\ep') \hookrightarrow E(a, ka +\ep)$ where $\lambda >0$ and $\ep, \ep'>0$ are arbitrarily small. Following the notations as above, fix the orbit sets $\alpha_2$ and $\beta_1^{k+1}$. Proposition \ref{prop-broken} implies that there exists a possibly broken $J$-holomorphic current $C$ in the corresponding cobordism $\overline{X}$ with $I(C)=0$. Indeed, 
\[ \frac{{\rm gr}(\alpha_2)}{2} = 1 + \floor*{\frac{ak+\ep}{a}} = 1+k \]
and 
\[ \frac{{\rm gr}(\beta_1^{k+1})}{2} = k+1 + \sum_{j=1}^{k+1} \floor*{\frac{j}{k+1+\ep'}} = k+1, \]
that is, ${\rm gr}(\alpha_2) = {\rm gr}(\beta_1^{k+1})$. It turns out that, for an appropriate $\lambda$, one can show that $C$ is a single somewhere injective curve (see Lemma \ref{lemma-initial-2}). By (\ref{j0-index}), 
\[ \frac{J_0(C)}{2} = - \sum_{j=1}^{k} \floor*{\frac{j}{k+1+\ep'}} = 0. \]
Then the inequality (\ref{j0-inequality}) says that $0 \geq 2(g(C) - 1 + \delta(C)) + \sum_{\gamma} (2n_{\gamma}-1)$. As we have ends on (possibly covers of) two Reeb orbits, $\alpha_2$ and $\beta_1$, the sum $\sum_{\gamma} (2n_{\gamma}-1)$ is at least $2$, with equality only if we have a single end asymptotic to each orbit. In other words, the orbit set $\beta_1^{k+1}$ is the $(k+1)$-fold cover of $\beta_1$. Moreover, since $g(C), \delta(C) \geq 0$, we must have equality and also $g(C) = \delta(C) =0$. Hence, $C$ is an embedded cylinder. 
\end{ex}

\section{SFT and Fredholm index} 

\subsection{SFT background} Symplectic field theory (SFT) is a modern machinery that defines algebraic invariants of symplectic cobordism via certain counting of pseudo-holomorphic curves with asymptotic ends. It was originally formulated in the fundamental work \cite{EGH00}. Since then, there has been extensive development on the foundations of SFT (see \cite{BEHWZ03,Hof06,Cas14,CM18}). For a detailed introduction of SFT, see \cite{Wen16}. One of the key ingredients in SFT is the Fredholm index. Explicitly, let $u: \dot\Sigma \to X$ be a finite energy $J$-holomorphic curve, where $\dot \Sigma$ is a punctured Riemannian sphere and $X$ is a four-dimensional symplectic cobordism equipped with a compatible almost complex structure $J$. Each puncture of $u$ is asymptotic to a Reeb orbit of $\partial X = \partial X^+ \sqcup \partial X^-$. Denote by $\{\gamma_i^+\}_{i=1}^{s_+}$ the collection of positive asymptotic orbits and $\{\gamma_i^-\}_{i=1}^{s_-}$ the collection of negative asymptotic orbits. We will always work in cases where these Reeb orbits are either nondegenerate or Morse-Bott, that is, they may come in smooth families. Denote by $S_i^+$ and $S_i^-$ these families (more precisely, the leaf spaces of the associated Morse-Bott submanifolds). Fix a symplectic trivialization $\tau$ of $u^*TX$ along these Reeb orbits, and $c_1^{\tau}(u^*TX)$ denotes the first Chern number with respect to $\tau$. Then 
\begin{align} \label{dfn-f-ind}
{\rm ind}(u) & = (s_+ + s_- - 2) + 2c_1^{\tau}(u^*TX)  \\
& \,\,\,\,\,+ \left(\sum_{i=1}^{s^+} {\rm CZ}^{\tau}(\gamma_i^+) + \frac{\dim S_i^+}{2}\right) - \left(\sum_{i=1}^{s_-} {\rm CZ}^{\tau}(\gamma_i^-) - \frac{\dim S_i^+}{2}\right). \nonumber
\end{align}
where ${\rm CZ}^{\tau}$ is the Robin-Salamon index with respect to $\tau$ (see \cite{RS93})  Note that ${\rm ind}(u)$ is independent of the choice of symplectic trivialization $\tau$. 

\medskip

There will be two concrete situations we will encounter in this paper. We summarize their properties in the following examples. 

\begin{ex} \label{ex-SFT} (1) Let $S^*_g \T^2$ denote the unit cosphere bundle of a 2-torus. With respect to the canonical contact structure, $S^*_g \T^2$ is a contact manifold and its Reeb orbits correspond to the closed geodesics on the base $(\T^2, g)$ via the projection from $S^*_g \T^2$ to the base manifold $\T^2$. We distinguish them by the homology classes of their corresponding closed geodesics. For instance, following the notation in Proposition 3.1 in \cite{HO19}, $\gamma_{(k,l)}$ denotes the Reeb orbit of $S^*_g \T^2$ whose projection in $\T^2$ is in the class $(k, l)$. Note that $\gamma_{(k,l)}$ is embedded if and only if $k, l$ are coprime. An easy but important observation is that none of these Reeb orbits is contractible in $S^*_g \T^2$. Moreover, each $\gamma_{(k,l)}$ is Morse-Bott, and the leaf space has dimension $\dim S = 1$.  

Lemma 3.1 in \cite{D-RGI16} says that for a symplectic cobordism $X$ equal to either the symplectization of $S^*_g \T^2$, or the full cotangent bundle $T^* \T^2$, by choosing the complex trivialization $\tau$ of the contact planes induced by complexifying the trivialization of $\T^2$, it follows that 
\begin{equation} \label{c1-cz-torus}
c_1^{\tau}(u^*TX) = 0\,\,\,\,\,\mbox{and}\,\,\,\,\, {\rm CZ}^{\tau}(\gamma_{(k,l)}) = \frac{1}{2}
\end{equation}
for any curve $u$ and Reeb orbit $\gamma_{(k,l)}$. On the other hand, if we consider the symplectic trivialization over $S^*_g \T^2$ that comes from the inclusion into $\C^2$, then 
\begin{equation} \label{c1-cz-torus-2}
c_1^{\tau}(u^*TX) = 0 \,\,\,\,\,\mbox{and}\,\,\,\,\, {\rm CZ}^{\tau}(\gamma_{(k,l)}) = 2(k+l) + \frac{1}{2}
\end{equation}
for any curve $u$ and Reeb orbit $\gamma_{(k,l)}$.

\medskip

(2) For ellipsoid $E(a,b)$ with $\frac{b}{a} \notin \Q$, any iterates of the two primitive Reeb orbits of the contact manifold $\partial E(a,b)$, say $\gamma^k_a$ and $\gamma^k_b$, are non-degenerate, so the corresponding Morse-Bott manifolds have dimension $\dim S =0$. For a symplectic cobordism $X$ equal to the completion of $E(a,b)$, or the symplectization of $\partial E(a,b)$, choosing the trivialization $\tau$ induced from the standard $\C^2$, it follows that $c_1^{\tau}(u^*TX) =0$, and (1), (2) in \cite{HK18-2} provide the following standard formula, 
\begin{equation} \label{cz-ellipsoid}
{\rm CZ}^{\tau}(\gamma_a^k) = 2k + 2\floor*{\frac{ka}{b}} +1 \,\,\,\,\,\mbox{and}\,\,\,\,\, {\rm CZ}^{\tau}(\gamma_b^k) = 2k + 2\floor*{\frac{kb}{a}} +1
\end{equation}
for any $k \in \N$. 
\end{ex}

As we have seen in Figure \ref{figure-climit}, to simplify the discussion, sometimes in split symplectic manifolds we abstractly glue preferred curves from different levels whenever they have matching Reeb orbits. Abstract gluing simply means that we think of the distinct curves as a single holomorphic building. The Fredholm index of the resulting building can be defined using formula (\ref{dfn-f-ind}) but summing only over the unmatched ends. This index can be computed using the Fredholm indices of the sub-curves recursively as follows (see Definition 3.2 and Proposition 3.3 in \cite{HO19}). Suppose $B$ is a building made of two curves $u_1$ and $u_2$ with matching ends $(\gamma_1, …, \gamma_l)$, then 
\begin{equation} \label{f-ind-glue}
{\rm ind}(B) = {\rm ind}(u_1) + {\rm ind}(u_2) - \sum_{i=1}^l \dim S_i 
\end{equation}
where $S_i$ is the corresponding Morse-Bott submanifold of $\gamma_i$. For instance, if the matching ends $(\gamma_1, …, \gamma_l)$ lies on $E(a,b)$ with $\frac{b}{a} \notin \Q$, then (2) in Example \ref{ex-SFT} and (\ref{f-ind-glue}) imply that ${\rm ind}(B) = {\rm ind}(u_1) + {\rm ind}(u_2)$. 

\subsection{Fredholm index computations} We call a finite energy $J$-holomorphic curve $u: \dot \Sigma \to X$ {\it rigid} if ${\rm ind}(u) =0$. This term justifies itself since the geometric meaning of the Fredholm index ${\rm ind}(u)$ is the virtual dimension of a moduli space that counts $J$-holomorphic curves with appropriate constraints. Therefore, an effective counting usually comes from $u$ with ${\rm ind}(u) \geq 0$. It is a well-known fact that if $u$ is somewhere injective, then for generic almost-complex structures we have ${\rm ind}(u) \geq 0$. In general, from (\ref{dfn-f-ind}), whether ${\rm ind}(u) \geq 0$ is satisfied heavily depends on the various combinations of the asymptotics. As previewed in the outline proof of the main result (see subsection \ref{ssec-outline}), eventually we will deal with different curves in symplectic cobordisms and symplectizations. To understand this, estimates of the Fredholm indices will be helpful. The following series of results make up a detailed study of Fredholm indices in the situations that will be useful later in the paper. 

\subsubsection{Symplectization $S\partial E(1, k+\ep)$ when $\ep$ is sufficiently small}

We assume $\ep$ is small in comparison to all covering numbers of the asymptotic limits of our curves.

\begin{lemma} \label{lemma-mix} Let $u$ be a $J$-holomorphic curve in $S\partial E(1, k+\ep)$ with $k \geq 2$. 
\begin{itemize}
\item[(1)] If $u$ has $m$-many positive ends of simple $\alpha_1$ and one negative end $\alpha_1^r$ for some winding number $r \geq 1$, then 
\[ {\rm ind}(u) \geq 2r - 2 - 2\floor*{\frac{r}{k+\ep}} \geq 0.\]
Moreover, ${\rm ind}(u) =0$ if and only if $u$ is a trivial cylinder with ends $\alpha_1$. 
\item[(2)] If $u$ has $m$-many positive ends of simple $\alpha_1$ and one positive end of simple $\alpha_2$; and one negative end $\alpha_1^r$ for some winding number $r \geq 1$, then 
\[ {\rm ind}(u) \geq 2r - 2k + 2 - 2\floor*{\frac{r}{k+\ep}}. \]
Moreover, ${\rm ind}(u) \geq m+1 \geq 2$. 
\item[(3)] If $u$ has $m$-many positive ends of simple $\alpha_1$ and one negative end $\alpha_2^r$ for some winding number $r \geq 1$, then 
\[ {\rm ind}(u) \geq 2r(2k-1) + 2 - 2\floor{r(k+\ep)} \geq 2r + 2.\]
Moreover, ${\rm ind}(u) = 2r+2$ if and only if $k =2$ and $m = 2r+1$.
\item[(4)] If $u$ has $m$-many positive ends of simple $\alpha_1$ and one positive end of simple $\alpha_2$; and one negative end $\alpha_2^r$ for some winding number $r \geq 1$, then if $m \neq 0$,
\[ {\rm ind}(u) \geq (4k-2)r -2k  + 2 - 2\floor{r(k+\ep)} \geq  2r+2. \]
Moreover, we have ${\rm ind}(u)= 2r+2$ if and only if $k=2$ and $m = 2r-1$. Finally, if $m=0$ then ${\rm ind}(u) = 0$ if and only if $u$ is a trivial cylinder with ends on $\alpha_2$. 
\end{itemize}
\end{lemma}

\begin{proof} (1) Consider curve $u$ in $S\partial E(1, k+ \ep)$ with $m$-many positive ends of simple $\alpha_1$ and one negative end $\alpha_1^r$ for some winding number $r \geq 1$. Then, by (\ref{dfn-f-ind}), 
\begin{align*}
{\rm ind}(u) & = (m+1-2) + m \left(2 + 2 \floor*{\frac{1}{k+\ep}} + 1\right) - \left(2r + 2 \floor*{\frac{r}{k+\ep}} +1\right)\\
& = (m-1) + 3m - 2r - 2 \floor*{\frac{r}{k+\ep}} - 1 \\
& = 4m - 2r - 2 \floor*{\frac{r}{k+\ep}} - 2.
\end{align*}
Meanwhile, the action difference is 
\[ \Delta_{\mathcal A} = m - r, \,\,\,\,\mbox{which is non-negative, that is, $m \geq r$}. \]
Therefore, we obtain the desired inequality ${\rm ind}(u) \geq 2r - 2 - 2\floor{r/(k+\ep)}$. On the other hand, since $\frac{r}{k+\ep} <r$ (due to $k \geq 1$), we know that $\floor*{\frac{r}{k+\ep}} \leq r-1$. Therefore, back to the Fredholm index, 
\begin{align*}
{\rm ind}(u) & \geq 2r - 2 - 2 \floor*{\frac{r}{k+\ep}} \geq 0.
\end{align*}
If ${\rm ind}(u) =0$, then 
\begin{equation} \label{cond-r}
2r - 2 - 2 \floor*{\frac{r}{k+\ep}} =0, \,\,\,\,\mbox{which is} \,\,\,\, r - 1-  \floor*{\frac{r}{k+\ep}} = 0.
\end{equation}
If $\floor*{\frac{r}{k+\ep}} \neq 0$, then $r \geq k + \ep \geq 2 + \ep$ by our assumption. This implies that 
\begin{align*}
r - 1- \floor*{\frac{r}{k+\ep}} & \geq r - 1- \frac{r}{k+\ep} \\
& = r \cdot \frac{k+ \ep-1}{k+\ep} - 1\\
& \geq (k+ \ep)  - 2\\
&  \geq (2 + \ep) - 2 > 0
\end{align*}
which contradicts the relation (\ref{cond-r}). Therefore, $\floor*{\frac{r}{k+\ep}} =0$, which implies, by (\ref{cond-r}) again, $r =1$. Moreover, ${\rm ind}(u) = 0$ also yields that $m = r$. Hence, $m =1$. This means that $u$ is a trivial cylinder with ends on $\alpha_1$. 

\medskip

(2) Consider curve $u$ in $S\partial E(1, k+ \ep)$ with $m$-many positive ends of simple $\alpha_1$ and one positive end of simple $\alpha_2$; and one negative end $\alpha_1^r$ for some winding number $r \geq 1$. Then, by (\ref{dfn-f-ind}), 
\begin{align*}
{\rm ind}(u) & = (m+1+1-2) + m \left(2 + 2 \floor*{\frac{1}{k+\ep}} + 1\right) + (2 + 2\floor*{k+\ep} +1) \\
& \,\,\,\,\,- \left(2r + 2 \floor*{\frac{r}{k+\ep}} +1\right)\\
& = m + 3m + (2k+3) - 2r - 2 \floor*{\frac{r}{k+\ep}} -1\\
& = 4m + 2k - 2r - 2 \floor*{\frac{r}{k+\ep}} + 2. 
\end{align*}
Meanwhile, the action difference is 
\begin{equation} \label{est-mid-2}
\Delta_{\mathcal A} = m + (k+ \ep) - r, \,\,\,\,\mbox{which is non-negative, that is, $m \geq r - (k+\ep)$}.
\end{equation}
This implies the desired inequality. Indeed, 
\begin{align*}
{\rm ind}(u) & = 4m + 2k - 2r - 2 \floor*{\frac{r}{k+\ep}} + 2\\
& \geq 4r - 4(k+ \ep) + 2k - 2r - 2 \floor*{\frac{r}{k+\ep}} + 2\\
& \geq 2r - 2k + 2 - 2\floor*{\frac{r}{k+\ep}}
\end{align*}
where the final step comes from the fact that ${\rm ind}(u)$ is an integer and $\ep$ is sufficiently small. On the other hand, (\ref{est-mid-2}) also implies that 
\begin{equation}\label{cond-r-2}
\floor*{\frac{r}{k+\ep}} \leq \frac{m}{k+\ep} + 1.
\end{equation}
Moreover, $k-r > -m -\ep$ (since $k + \ep$ is irrational, we will not have equality), which is equivalent to $k- r \geq -m$. Now, back to the Fredholm index, (\ref{cond-r-2}) implies that
\begin{align*}
{\rm ind}(u) & = 4m + 2(k -r) - 2 \floor*{\frac{r}{k+\ep}} + 2 \\
& \geq 4m + (-2m) - 2 \left(\frac{m}{k+\ep} + 1\right) + 2\\
& = 2m - \frac{2m}{k+\ep} > 2m - \frac{2m}{2}  = m
\end{align*}
where the second last step uses the hypothesis that $k \geq 2$. Hence, ${\rm ind}(u) \geq m+1$. Note that when $m =0$, we have ${\rm ind}(u) \geq 2$ since ${\rm ind}(u)$ is always even. In particular, ${\rm ind}(u)$ is never equal to $0$ in this case. 

\medskip

(3) Consider curve $u$ in $S\partial E(1, k+ \ep)$ with $m$-many positive ends of simple $\alpha_1$ and one negative end $\alpha_2^r$ for some winding number $r \geq 1$. Then, by (\ref{dfn-f-ind}), 
\begin{align*}
{\rm ind}(u)& = (m+1-2) + m \left(2 + 2 \floor*{\frac{1}{k+\ep}}+1\right) - (2r + 2\floor*{r(k+\ep)} +1) \\
& = m-1+3m - 2r -2\floor*{r(k+\ep)} -1\\
& = 4m -2r - 2\floor*{r(k+\ep)} -2.
\end{align*}
Meanwhile, the action difference is 
\[ \Delta_{\mathcal A} = m - r(k+\ep) \,\,\,\,\mbox{which is non-negative, that is, $m > rk + r\ep$}\]
where we will not have equality since $k+\ep$ is irrational. For sufficiently small $\ep$, this condition is equivalent to $m \geq rk+1$. Then it implies the desired inequality. Indeed, 
\begin{align*}
{\rm ind}(u)& = 4m -2r - 2\floor*{r(k+\ep)} -2 \\
& \geq 4rk + 4 - 2r - 2\floor*{r(k+\ep)} -2 \\
& \geq  2r(2k-1) + 2 - 2\floor{r(k+\ep)}.
\end{align*}
Now, when $\ep$ is sufficiently small, $\floor*{r(k+\ep)} = rk$. Now, back to the Fredholm index, 
\begin{align*}
{\rm ind}(u) & \geq 2r(2k-1) + 2 - 2\floor{r(k+\ep)} \\ 
& = 4rk - 2r + 2 - 2rk\\
& \geq 2rk - 2r + 2 \geq 2r + 2
\end{align*}
where the last step comes from the hypothesis that $k \geq 2$. Note that ${\rm ind}(u) = 2r+2$ if and only if $k =2$ and $m = 2r+1$. 

\medskip

(4) Consider curve $u$ in $S\partial E(1, k+ \ep)$ with $m$-many positive ends of simple $\alpha_1$ and one positive end of simple $\alpha_2$; and one negative end $\alpha_2^r$ for some winding number $r \geq 1$. Then, by (\ref{dfn-f-ind}), 
\begin{align*}
{\rm ind}(u) & = (m+1+1-2) + m \left(2 + 2 \floor*{\frac{1}{k+\ep}} + 1\right) + (2 + 2\floor*{k+\ep} +1) \\
& \,\,\,\,\,- \left(2r + 2 \floor*{r(k+\ep)} +1\right)\\
& = m + 3m + (3  +2\floor*{k+\ep}) - (2r + 2\floor*{r(k+\ep)} +1) \\
& = 4m + (2k+3) - (2r + 2\floor*{r(k+\ep)}  +1) \\
& = 4m + 2k - 2r - 2\floor*{r(k+\ep)} + 2
\end{align*}
where in the third step we use the assumption that $\ep$ is sufficiently small. Meanwhile, the action difference is 
\begin{equation} \label{est-mid-4}
\Delta_{\mathcal A} = m + (k+\ep) -  r(k+\ep) \,\,\,\,\mbox{which is non-negative, so $m \geq (r-1)(k + \ep)$}
\end{equation}
where we have equality if and only if $m = 0$ and $r=1$ (since $k+\ep$ is irrational). Then let us carry on the discussion in two cases. 
\begin{itemize}
\item[(i)] If $m \neq 0$, then (\ref{est-mid-4}) yields the desired inequality. Indeed,
\begin{align*}
{\rm ind}(u) & = 4m + 2k - 2r - 2\floor*{r(k+\ep)} + 2\\
& \geq (4k+4\ep-2)r - 4(k+\ep) + 2k +2 - 2\floor*{r(k+\ep)} \\
& \geq (4k-2) r - 2k + 2 - 2\floor*{r(k+\ep)}.
\end{align*}
On the other hand, for sufficiently small $\ep$ we have $\floor*{r(k+\ep)} = rk$. Now, back to the Fredholm index, 
\begin{align*}
{\rm ind}(u)  & \geq  (4k-2) r - 2k + 2 - 2\floor*{r(k+\ep)} \\
 & = (4k-2)r - 2k + 2 - 2rk \\
 & = (2k-2)(r-1) + 4 \geq 2r + 2
\end{align*}
where the final step uses the hypothesis that $k \geq 2$. Note that ${\rm ind}(u)= 2r+2$ if and only if $k=2$ and $m = 2r-1$. 
\item[(ii)] If $m =0$ (then $r=1$), then the curve $u$ is a trivial cylinder with ends on $\alpha_2$ and ${\rm ind}(u) =0$, vice versa. 
\end{itemize}
Therefore, we complete the proof. \end{proof}

\subsubsection{Symplectic cobordism $\overline{X}$ from inclusion $E(\ep, \ep S) \hookrightarrow E(1,k+\ep)$} 

\begin{lemma} \label{lemma-2} Let $u$ be a $J$-holomorphic curve in the symplectic cobordism part $\overline{X}$ with positive ends $(\alpha_1^{r_1}, …, \alpha_1^{r_{n_1}}, \alpha_2^{s_1}, …, \alpha_2^{s_{n_2}})$ and negative ends $(\beta_1^{t_1}, …, \beta_1^{t_{n_3}}, \beta_2^{u_1}, …, \beta_2^{u_{n_4}})$. Then its Fredholm index is 
\begin{align*}
{\rm ind}(u) & = 2n_1 + 2n_2 - 2 + 2 \sum_{i=1}^{n_1} \left(r_i + \floor*{\frac{r_i}{k+\ep}}\right) + 2 \sum_{i=1}^{n_2} \left(s_i + \floor*{s_i(k+\ep)}\right)\\
& - 2 \sum_{i=1}^{n_3} \left(t_i + \floor*{\frac{t_i}{S}} \right) - 2 \sum_{i=1}^{n_4} \left(u_i + \floor*{u_i S}\right).
\end{align*}
\end{lemma}

\begin{proof} This directly comes from (\ref{dfn-f-ind}) and (\ref{cz-ellipsoid}). \end{proof}

Suppose $u$ is an $m$-fold cover of a somewhere injective curve $\tilde{u}$ (with asymptotics now denoted $\tilde{\cdot}$). Then Lemma \ref{lemma-2} says that 
\begin{align*}
{\rm ind}(\tilde{u}) & = 2\tilde{n}_1 + 2\tilde{n}_2 - 2 + 2 \sum_{i=1}^{\tilde{n}_1} \left(\tilde{r}_i + \floor*{\frac{\tilde{r}_i}{k+\ep}}\right) + 2 \sum_{i=1}^{\tilde{n}_2} \left(\tilde{s}_i + \floor*{\tilde{s}_i(k+\ep)}\right)\\
& - 2 \sum_{i=1}^{\tilde{n}_3} \left(\tilde{t}_i + \floor*{\frac{\tilde{t}_i}{S}} \right) - 2 \sum_{i=1}^{\tilde{n}_4} \left(\tilde{u}_i + \floor*{\tilde{u}_i S}\right).
\end{align*}
Moreover, $\sum r_i = m \sum \tilde{r}_i$ and similarly for the other ends. Then we have the following corollary. 

\begin{cor} \label{cor-1} If $u$ is an $m$-fold cover of a somewhere injective curve $\tilde{u}$ for a generic almost complex structure, then 
\begin{align*}
{\rm ind}(u) & \geq (2m-2) + 2(n_1 - m \tilde{n}_1) + 2(n_2 - m \tilde{n}_2) \\ 
& + 2 \left(\sum_{i=1}^{n_1}  \floor*{\frac{r_i}{k+\ep}} - m \sum_{i=1}^{\tilde{n}_1} \floor*{\frac{\tilde{r}_i}{k+\ep}} \right) + 2 \left(\sum_{i=1}^{n_2}  \floor*{s_i(k+\ep)} - m \sum_{i=1}^{\tilde{n}_2} \floor*{s_i(k+\ep)}\right)\\
& - 2 \left(\sum_{i=1}^{n_3}  \floor*{\frac{t_i}{S}} - m \sum_{i=1}^{\tilde{n}_3} \floor*{\frac{\tilde{t}_i}{S}} \right) - 2 \left(\sum_{i=1}^{n_4}  \floor*{u_i S} - m \sum_{i=1}^{\tilde{n}_4} \floor*{\tilde{u}_i S}\right)
\end{align*}
\end{cor} 

\begin{proof} This is from the fact that ${\rm ind}(\tilde{u}) \geq 0$ since $\tilde{u}$ is a somewhere injective curve. \end{proof}

\subsubsection{Symplectization $S\partial E(1, S)$}

\begin{lemma} \label{lemma-se-bottom} Let $u$ be a $J$-holomorphic curve in the symplectization $S\partial E(1, S)$ with positive ends $(\beta_1^{r_1}, …, \beta_1^{r_{n_1}}, \beta_2^{s_1}, …, \beta_2^{s_{n_2}})$ and a negative end $\beta_1^{t}$ with $S >t$. Then ${\rm ind}(u) \geq 0$, and ${\rm ind}(u)=0$ if and only if $u$ has a single positive end $\beta_1^t$. \end{lemma}

\begin{proof} For the curve $u$ in the hypothesis, by (\ref{dfn-f-ind}), 
\begin{align} 
\nonumber {\rm ind}(u) & = (n_1+n_2 +1 - 2) + \sum_{i=1}^{n_1} \left(2r_i + 2 \floor*{\frac{r_i}{S}} + 1\right) + \sum_{i=1}^{n_2} \left(2s_i + 2\floor*{s_i S} + 1\right) \\
\nonumber &\,\,\,\,\,\,\,\,- \left(2t + 2\floor*{\frac{t}{S}} +1\right) \\
\label{Wform} &  = 2n_1 + 2n_2 -2 + 2 \sum_{i=1}^{n_1} \left(r_i + \floor*{\frac{r_i}{S}}\right) + 2 \sum_{i=1}^{n_2} \left(s_i + \floor*{s_i S} \right) - 2t. 
\end{align}
Meanwhile, the action difference
\begin{equation*}
\Delta_{\mathcal A} = \sum_{i=1}^{n_1} r_i + \sum_{i=1}^{n_2} s_i S - t  \geq 0.
\end{equation*} 
Note that we have $\Delta_{\mathcal A} =0$ only when $n_2 =0$ since $S$ is assumed to be irrational. Observe that since at least one of $n_1, n_2$ is positive, we have $2n_1 + 2n_2 - 2 \geq 0$. Also, for each $i$, the term $s_i + \floor*{s_i S} \geq s_i S$ and $s_i + \floor*{s_i S} = s_i S$ if and only if $s_i = 0$. Now, back to the Fredholm index of $u$ as above, 
\begin{align*}
{\rm ind}(u) & = (2n_1 + 2n_2 -2) + 2 \sum_{i=1}^{n_1} \left(r_i + \floor*{\frac{r_i}{S}}\right) + 2 \sum_{i=1}^{n_2} \left(s_i + \floor*{s_i S} \right) - 2t\\
& \geq 0 + 2 \left( \sum_{i=1}^{n_1} \floor*{\frac{r_i}{S}} + \sum_{i=1}^{n_1} r_i + \sum_{i=1}^{n_2} \left(s_i + \floor*{s_i S}\right) - t\right) \\
& \geq 0 + 2 \left( \sum_{i=1}^{n_1} r_i + \sum_{i=1}^{n_2} s_i S - t \right) \geq 0.
\end{align*} 
Moreover, ${\rm ind}(u) =0$ if and only if $n_2 = 0$, $n_1 = 1$ and $r_1 = t$. The conclusion $n_2 =0$ comes from $\Delta_{\mathcal A} =0$. The conclusion $n_1 =1$ comes from $2n_1 + 2n_2 -2 =0$. The conclusion $r_1 = t$ comes from 
\[ r_1 + \floor*{\frac{r_1}{S}} -t =0, \,\,\,\,\mbox{which implies that} \,\,\,\, r_1 < S \]
since $S > t$, and then $r_1 = t$. \end{proof}

\section{Compactness and gluing} \label{sec-compact-glue}

This section will produce a family of rigid pseudo-holomorphic curves that will be used later in Section \ref{sec-proof} for the obstructions of Theorem \ref{thm-2}. 

\subsection{Compactness}
We will start with a compactness result. To this end, let us recall some notations. Let $E(a,b) = E(a, ka + \ep)$ and $E(\ep, \ep S)$ be ellipsoids with $k\geq 2$, $S \notin \Q$, and $\ep$ arbitrarily small such that $ka+\ep \notin \Q$ and $E(\ep, \ep S) \subset E(a,b)$. Let $\overline{X}$ be the completion of the symplectic cobordism from $E(a,b)$ to $E(\ep, \ep S)$. Moreover, denote by $\alpha_1$ and $\alpha_2$ the primitive Reeb orbits of $E(a, ka+\ep)$ with action $a$ and $ka + \ep$ respectively; denote by $\beta_1$ and $\beta_2$ the primitive Reeb orbits of $E(\ep, \ep S)$ with action $\ep$ and $\ep S$ respectively. For a generic almost complex structure $J$ of $\overline{X}$ with cylinder ends, consider the following moduli spaces, where all curves really denote equivalence classes under reparameterizations of the domain:
\begin{equation} \label{dfn-moduli-2}
\mathcal M^s_{J, \ep, S}(\underbrace{\alpha_1, …, \alpha_1}_{{\tiny \mbox{$m$-many}}}, \alpha_2; \beta_1^t) := {\small \left\{\begin{array}{l} \mbox{simple $J$-holomorphic curve $u$ with positive} \\ \mbox{ends $(\alpha_1, …, \alpha_1, \alpha_2)$ and negative end $\beta_1^t$}  \end{array} \right\}}.
\end{equation}
Here, $m \in \N$ and $S>t$, where $t$ is chosen so that the curves are rigid, that is, $t=2m+k+1$. Now, for $1$-parameter families $\{J_{\tau} \}$ consider the corresponding moduli spaces as follows, which have virtual dimension $1$. 
\begin{equation} \label{dfn-moduli-I}
\mathcal M_{m} : = \left\{(u,\, \tau) \,\bigg| \, u \in \mathcal M^s_{J_{\tau}, \ep, S}(\underbrace{\alpha_1, …, \alpha_1}_{{\tiny \mbox{$m$-many}}}, \alpha_2; \beta_1^t), \,\,\tau \in [0,1]\right\}. 
\end{equation}
For simplicity, we will use $\ep$ to uniformly denote an arbitrarily small number. 

\begin{theorem} \label{thm-compact} For generic $\{J_{\tau} \}$, with $\tau$ in a compact interval, the moduli space $\mathcal M_{m}$ is sequentially compact. \end{theorem}

The proof of Theorem \ref{thm-compact} is lengthy, and we follow the idea of the proof of Theorem 4.7 in Section 4 in \cite{C-GH18} but now a more detailed analysis of the limit holomorphic building is necessary. To this end, we need some preparations.

\begin{lemma} \label{lemma-floor-ceil}
For any $a \geq 0$ and $\lambda \in \N$, we have 
\begin{itemize}
\item[(i)] $\floor*{a} - \lambda \floor*{\frac{a}{\lambda}} \geq 0$;
\item[(ii)] $\ceil*{a} - \lambda \ceil*{\frac{a}{\lambda}} \geq - \lambda + 1$.
\end{itemize}
\end{lemma}

\begin{proof} (i) By Hermite equality, we have 
\[ \floor*{a} = \floor*{\lambda \cdot \frac{a}{\lambda}} = \floor*{\frac{a}{\lambda}} + \floor*{\frac{a+1}{\lambda}} + … + \floor*{\frac{a+\lambda-1}{\lambda}}. \]
For each $i \in \{0, …, \lambda-1\}$, we have $\floor*{\frac{a+i}{\lambda}}\geq \floor*{\frac{a}{\lambda}}$. Therefore, 
\[ \floor*{a} \geq \lambda \cdot \floor*{\frac{a}{\lambda}}, \]
which implies the desired conclusion.

(ii) By Hermite equality, we have 
\[ \ceil*{a} = \ceil*{\lambda \cdot \frac{a}{\lambda}} = \ceil*{\frac{a}{\lambda}} + \ceil*{\frac{a-1}{\lambda}} + … + \ceil*{\frac{a-(\lambda-1)}{\lambda}}. \]
For each $i \in \{1, …, \lambda-1\}$, we have $\ceil*{\frac{a-i}{\lambda}} \geq \ceil*{\frac{a}{\lambda}} -1$. Therefore, 
\[ \ceil*{a} \geq \lambda \cdot \ceil*{\frac{a}{\lambda}} - (\lambda -1), \]
which implies the desired conclusion. 
\end{proof}

For a sequence of ${J_{\tau_n}}$-holomorphic curves $\{(u_n, \tau_n)\}_{n \geq 1}$ in $\mathcal M_{m}$, by the SFT compactness in \cite{BEHWZ03}, a subsequence converges to a holomorphic building consisting of curves in $S\partial E(a,b)$, $\overline{X}$ and $S \partial E(\ep, \ep S)$. 
Instead of analyzing each individual curve in such a limiting building separately, we make certain identifications to treat various unions of curves with matching asymptotics together as a single connected component.
By the convention in \cite{C-GH18}, a component $C$ in $\overline{X}$ consists of 
\begin{itemize}
\item{} curves in $\overline{X}$, denoted by $\{u^p\}_{1 \leq p \leq P}$;
\item{} components in $S\partial E(\ep,\ep S)$ without negative ends, denoted by $\{W^q\}_{1\leq q \leq Q}$, which themselves consist of connected matching unions of curves, and whose unmatched positive ends match with negative ends of the $\{u^p\}_{1\leq p \leq P}$.
\end{itemize}
Further the connected component $C$ has a single unmatched negative end, which we denote to be an end of $u^1$. We also use the same notation for the asymptotic limits as in Lemma \ref{lemma-2}, namely
\begin{itemize}
\item{} each $u^p$ has positive ends $\{\alpha_1^{r^p_1}, …, \alpha_1^{r^p_{n_1}}, \alpha_2^{s^p_1}, …, \alpha_2^{s^p_{n_2}}\}$;
\item{} each $u^p$ has negative ends $\{\beta_1^{t^p_1}, …, \beta_1^{t^p_{n_3}}, \beta_2^{u^p_1}, …, \beta_2^{u^p_{n_4}}\}$. 
\end{itemize}

Also, we again assume $u^p$ is an $m^p$-fold cover of a somewhere injective curve $\tilde{u}^p$ for each $p \in \{1, …, P\}$, and follow the notation before Corollary \ref{cor-1} for this underlying curve. Note that, for genus reasons, $\sum_{p} n_3^p + \sum_{q} n_4^p = P + Q$.  We are interested in the Fredholm index of $C$. By matching-index formula (\ref{f-ind-glue}),  
\begin{equation}\label{C-sum}
{\rm ind}(C) = \sum_{p=1}^P {\rm ind}(u^p) + \sum_{q=1}^{Q} {\rm ind}(W^q).
\end{equation}

The remaining curves of the limiting holomorphic building can be identified as follows. We let $C_b$ be the connected union of curves in $S\partial E(\ep,\ep S)$ with matching ends where the unmatched positive ends correspond to the unmatched negative ends of the $C_i$, the components in $\overline{X}$. Then $C_b$ has a single unmatched negative end at $\beta_1^t$ (as the original curve).


Finally, denote by $\{v_j\}_{1 \leq j \leq J}$ for some $J \in \N$ the connected matched unions of the remaining curves in our building. For genus reasons, each $v_j$ has a single unmatched negative end corresponding to a positive end of one of the $C_i$ and, by the maximum principle, must have unmatched positive ends at covers of $\alpha_1$ or $\alpha_2$.

\medskip

We start with the following.

\begin{lemma}\label{noends} The $u^p$ have no negative ends on $\beta_2$.
\end{lemma}

\begin{proof} Ignoring the superscripts, the index of the underlying somewhere injective curve $\tilde{u}$ is given by
\begin{align*}
{\rm ind}(\tilde{u}) & = - 2 + 2 \sum_{i=1}^{\tilde{n}_1} \left(\tilde{r}_i + \ceil*{\frac{\tilde{r}_i}{k+\ep}}\right) + 2 \sum_{i=1}^{\tilde{n}_2} \left(\tilde{s}_i + \ceil*{\tilde{s}_i(k+\ep)}\right)\\
& - 2 \sum_{i=1}^{\tilde{n}_3} \left(\tilde{t}_i + \floor*{\frac{\tilde{t}_i}{S}} \right) - 2 \sum_{i=1}^{\tilde{n}_4} \left(\tilde{u}_i + \floor*{\tilde{u}_i S}\right) \ge 0.
\end{align*}
Now, only the first two sums here give a positive contribution, and this is bounded by the corresponding sums in the index formula for $u$ (as each positive end of $\tilde{u}$ is covered by a positive end of $u$). Summing over $p$, an upper bound is
$$2 \sum_{i,p} \left(r^p_i + \ceil*{\frac{r^p_i}{k+\ep}}\right) + 2 \sum_{i,p} \left(s^p_i + \ceil*{s^p_i(k+\ep)}\right).$$

Meanwhile, the index of a $v_j$ with $m$ positive ends simply covering $\alpha_1$ and $\kappa \in \{0,1\}$ positive ends on $\alpha_2$ and a single negative end covering $\alpha_1^r$ is given by
$${\rm ind}(v_j) = 4m + 2\kappa(2+k) - 2\left(r + \ceil*{\frac{r}{k+\ep}}\right)$$
and if the negative end covers $\alpha_2^s$ the index is
$${\rm ind}(v_j) = 4m + 2\kappa(2+k) - 2(r + \ceil*{s(k+\ep)}).$$
By Lemma \ref{lemma-mix} our $v_j$ all have nonnegative index. Hence, summing over all $p$ and $j$ the upper bound on the positive terms in ${\rm ind}(\tilde{u})$ is bounded above in turn by $4m + 2(2+k)$ where the $m$ now refers to the asymptotic limits of our original curve. But this curve has index $0$ and a single negative end on $\beta_1^t$ and so
$$4m + 2\kappa(2+k) = 2\left(t + \ceil*{\frac{t}{S}}\right) < 2S$$ as $t<S$. The result follows.
\end{proof}

The following lemma is crucial.

\begin{lemma} \label{lemma-hol-build} Let $C$ be a component in $\overline{X}$ whose unmatched negative end lies at $\beta_1^{t^1_1}$. Then
${\rm ind}(C) \geq 2 \sum_{p} \left(n^p_2 - m^p \tilde{n}^p_2\right) - 2\floor*{\frac{t^1_1}{S}}$. There is equality only if there are no $W^q$.

\end{lemma}

\begin{proof} For the first term in (\ref{C-sum}), by Corollary \ref{cor-1},  
\begin{align*}
\frac{\sum_p {\rm ind}(u^p)}{2} & \geq \left(\sum_{p} m^p -P\right) + \sum_p \left(n^p_1 - m^p \tilde{n}^p_1\right) + \sum_{p} \left(n^p_2 - m^p \tilde{n}^p_2\right) \\ 
& + \left(\sum_{i, p}  \floor*{\frac{r^p_i}{k+\ep}} - \sum_{i, p} m^p \floor*{\frac{\tilde{r}^p_i}{k+\ep}} \right) + \left(\sum_{i, p} \floor*{s^p_i(k+\ep)} - \sum_{i,p}m^p \floor*{\tilde{s}^p_i(k+\ep)}\right)\\
& -  \left(\sum_{i, p}  \floor*{\frac{t^p_i}{S}} - \sum_{i,p} m^p\floor*{\frac{\tilde{t}^p_i}{S}} \right). 
\end{align*}
Here we are using Lemma \ref{noends} to drop terms involving negative ends asymptotic to $\beta_2$.

Let us focus on the sums with $s_i^p$ and $\tilde{s}^p_i$. Note that each end of $u^p$ corresponds to a single end of $\tilde{u}^p$, and if $s^p_j$ corresponds to $\tilde{s}^p_1$, say, then we have $s^p_j = \lambda_j \tilde{s}^p_1$ for some $\lambda_j \in \N$. Then by Lemma \ref{lemma-floor-ceil} (i), 
\begin{equation*}
\floor*{s^p_j(k+\ep)}  = \floor*{\lambda_j \tilde{s}^p_1(k+\ep)} \geq \lambda_j \floor*{\tilde{s}^p_1(k+\ep)}. 
\end{equation*}
For each $p$, summing up over all $j$ such that $s^p_j$ corresponds to $\tilde{s}^p_1$, we have $\sum_j \lambda_j = m^p$. Therefore, still for each $p$, summing up over all $\tilde{s}^p_i$ for $i \in \{1, …, \tilde{n}^p_2\}$, we have 
\begin{equation} \label{quote1}
\sum_{j=1}^{n^p_2} \floor*{s^p_j(k+\ep)} \geq m^p \sum_{i=1}^{\tilde{n}^p_2} \floor*{\tilde{s}^p_i(k+\ep)}.
\end{equation}

Finally we sum over all $p$, and we get $\sum_{i, p} \floor*{s^p_i(k+\ep)} - \sum_{i,p}m^p \floor*{\tilde{s}^p_i(k+\ep)} \geq 0$. 

Then by adding $\sum_{q=1}^{Q} {\rm ind}(W^q)$ as in (\ref{Wform}), we get 
\begin{align*}
\frac{{\rm ind}(C)}{2} & \geq \left(\sum_{p} m^p -P\right) +  \sum_p \left(n^p_1 - m^p \tilde{n}^p_1\right) + \sum_{p} \left(n^p_2 - m^p \tilde{n}^p_2\right) \\ 
& + \left(\sum_{i, p}  \floor*{\frac{r^p_i}{k+\ep}} - \sum_{i, p} m^p \floor*{\frac{\tilde{r}^p_i}{k+\ep}} \right) + \left(\sum_{i,p} m^p\floor*{\frac{\tilde{t}^p_i}{S}}-  \floor*{\frac{t_1^1}{S}}\right) \\
& + \sum_{p} n_3^p  - Q + \sum_{(i,p) \neq (1,1)} t^p_i. 
\end{align*}

Now, $\sum_{p} n_3^p + 1 = P+Q$. Therefore, we can simplify the estimation as follows,
\begin{align*}
\frac{{\rm ind}(C)}{2} & \geq \sum_{p} m^p + \sum_p \left(n^p_1 - m^p \tilde{n}^p_1\right) + \sum_{p} \left(n^p_2 - m^p \tilde{n}^p_2\right) \\
& \,\,\,\,\,+ \left(\sum_{i, p}  \floor*{\frac{r^p_i}{k+\ep}} - \sum_{i, p} m^p \floor*{\frac{\tilde{r}^p_i}{k+\ep}} \right) + \left(\sum_{i, p} m^p\floor*{\frac{\tilde{t}^p_i}{S}}-  \floor*{\frac{t_1^1}{S}}\right) -1.
\end{align*}
Note that if no ends are asymptotic to $\alpha_1$ then our required inequality follows immediately, as the second and fourth terms vanish. Otherwise we proceed as follows.

For each $p$, since $\floor*{\frac{r_i^p}{k+\ep}}$ is irrational, we know $1+\floor*{\frac{r_i^p}{k+\ep}} = \ceil*{\frac{r_i^p}{k+\ep}}$. Therefore, 
\begin{align*}
\sum_p n^p_1 +  \sum_{i, p}  \floor*{\frac{r^p_i}{k+\ep}} & = \sum_p \sum_{i=1}^{n^p_1} \left(1 +  \floor*{\frac{r^p_i}{k+\ep}}\right) = \sum_p \sum_{i=1}^{n^p_1} \ceil*{\frac{r^p_i}{k+\ep}}. 
\end{align*}
Now, similarly to the above, assume the end $r^p_j$ of $u^p$ corresponds to $\tilde{r}^p_1$. Then  $r^p_j = \lambda_j \tilde{r}^p_1$ and by (ii) in Lemma \ref{lemma-floor-ceil}, 
\begin{align*}
\ceil*{\frac{r^p_j}{k+\ep}} & = \ceil*{\frac{\lambda_j \tilde{r}^p_1}{k+\ep}} \geq \lambda_j \ceil*{\frac{\tilde{r}^p_1}{k+\ep}} - \lambda_j + 1. 
\end{align*}
For each $p$, summing up over all $j$ such that $r^p_j$ corresponds to $\tilde{r}^p_1$, we have $\sum_j \lambda_j = m^p$. Therefore, for each $p$, summing over all $\tilde{r}^p_i$ for $i \in \{1, …, \tilde{n}^p_1\}$, we have
\begin{equation} \label{quote2}
\begin{split}
\sum_{j=1}^{n^p_1} \ceil*{\frac{r^p_j}{k+\ep}} & \geq m^p \sum_{i=1}^{\tilde{n}_1^p} \ceil*{\frac{\tilde{r}^p_i}{k+\ep}} - m^p + n_1^p \\
& = m^p \sum_{i=1}^{\tilde{n}_1^p} \left(1+ \floor*{\frac{\tilde{r}^p_i}{k+\ep}}\right) - m^p + n_1^p\\
& = m^p \sum_{i=1}^{\tilde{n}_1^p} \floor*{\frac{\tilde{r}^p_i}{k+\ep}} + m^p \tilde{n}_1^p -m^p + n_1^p
\end{split}
\end{equation}
where the first equality comes from the fact that $\ceil*{\frac{\tilde{r}^p_i}{k+\ep}}  = 1 +\floor*{\frac{\tilde{r}^p_i}{k+\ep}}$ since $\frac{\tilde{r}^p_i}{k+\ep}$ is irrational. Then, sum up over all $p$, we get 
\[ \sum_p \sum_{i=1}^{n^p_1} \ceil*{\frac{r^p_i}{k+\ep}} \geq \sum_p \sum_{i=1}^{\tilde{n}_1^p} m^p \floor*{\frac{\tilde{r}^p_i}{k+\ep}} + \sum_{p} m^p \tilde{n}_1^p  - \sum_{p} m^p + \sum_{p} n_1^p.  \]
Then, back to the Fredholm index of $C$, we have 
\begin{align*}
\frac{{\rm ind}(C)}{2} & \geq  \sum_{p} \left(n^p_2 - m^p \tilde{n}^p_2\right) + \sum_{p} n_1^p + \left(\sum_{i,p} m^p\floor*{\frac{\tilde{t}^p_i}{S}}-  \floor*{\frac{t_1^1}{S}}\right) -1\\
& \geq  \sum_{p} \left(n^p_2 - m^p \tilde{n}^p_2\right) + \left(\sum_{i,p} m^p\floor*{\frac{\tilde{t}^p_i}{S}}-  \floor*{\frac{t_1^1}{S}}\right)\\
& \geq \sum_{p} \left(n^p_2 - m^p \tilde{n}^p_2\right) - \floor*{\frac{t_1^1}{S}}.
\end{align*}
The second inequality is due to the fact that at least one $n_1^p$ is positive since we have ends asymptotic to $\alpha_1$ (so $\sum_{p} n_1^p \geq 1$). Thus we complete the proof. \end{proof}

\begin{remark} \label{rmk-lower-bound}
We observe that the term $\sum_{p} \left(n^p_2 - m^p \tilde{n}^p_2\right)$ in Lemma \ref{lemma-hol-build} is bounded from below by $-\sum_{i, p}(s^p_i-1)$, where $s^p_i$ is the winding numbers over $\alpha_2$ of the component $u^p$ of the curve $C$. It is an equality if and only if all $\tilde{s}^p_i= 1$.
\end{remark}

Now, we are ready to give the proof of Theorem \ref{thm-compact}. 

\begin{proof} [Proof of Theorem \ref{thm-compact}] By the SFT compactness in \cite{BEHWZ03}, the limit of curves in $\mathcal M_{m}$ is a holomorphic building with components in levels $S\partial E(a, ka+\ep)$, $\overline{X}$ and $S\partial E(\ep, \ep S)$. Assume that there are $I$-many curves in the symplectic cobordism level $\overline{X}$, after making the identifications as above, which we denote by $C_1, …, C_I$. For each $1 \le i \le I$ let the unmatched end of $C_i$ cover $\beta_1$ with multiplicity $t^{i,1}_{1}$.


By hypothesis, the initial curve has Fredholm index $0$, so we have 
\begin{equation} \label{ind=0}
\sum_{j=1}^J {\rm ind}(v_j) + \sum_{i=1}^I {\rm ind}(C_i) + {\rm ind}(C_b) =0.
\end{equation}
Since the total winding number of the negative ends of all $\{v_j\}_{1 \leq j \leq J}$ is equal to the total winding number of the positive ends of all $\{C_i\}_{1 \leq i \leq I}$, Lemma \ref{lemma-mix}, Lemma \ref{lemma-hol-build} and Remark \ref{rmk-lower-bound} together imply that 
\begin{equation} \label{ind>0}
\sum_{j=1}^J {\rm ind}(v_j) + \sum_{i=1}^I {\rm ind}(C_i) + {\rm ind}(C_b) \geq - 2\sum_{i=1}^I \floor*{\frac{t^{i,1}_1}{S}} + {\rm ind}(C_b). 
\end{equation}

{\it Claim.} For any $i \in \{1, …, I\}$, the winding number $t^{i,1}_1 <S$.  

{\it Proof of Claim.} By Lemma \ref{noends}  there are no ends asymptotic to $\beta_2$.
By assumption therefore, $C_b$ has one negative end $\beta_1^t$ (as the original curve) and positive ends $\{\beta_1^{r_1}, … \beta_1^{r_{l}}\}$, where $r_i = t^{i,1}_1$ for any $i \in \{1, …, I\}$ (recall that $I = l$). Suppose, without loss of generality, ${t^{1,1}_1} \geq S$, then by our hypothesis, $t^{1,1}_1 >t$ since $S>t$. By the index formula (\ref{dfn-f-ind}), 
\begin{align*}
\frac{{\rm ind}(C_b)}{2} &  = (l - 1) + \sum_{i=1}^{l} r_i + \sum_{i=1}^{l} \floor*{\frac{r_i}{S}}    - t. \\
& \ge 1 + \sum_{i=1}^{l} \floor*{\frac{r_i}{S}}.
\end{align*}


Therefore, back to the estimation (\ref{ind>0}), we get $\sum_{j=1}^J {\rm ind}(v_j) + \sum_{i=1}^I {\rm ind}(C_i) + {\rm ind}(C_b) \geq 2$, which contradicts (\ref{ind=0}). Thus we finish the proof of the claim.

Now, our claim  improves the lower bound of (\ref{ind>0}) to be ${\rm ind}(C_b)$ and $\sum_{j=1}^J {\rm ind}(v_j) + \sum_{i=1}^I {\rm ind}(C_i) \geq 0$. Thus by (\ref{ind=0}) we have ${\rm ind}(C_b)=0$, and by Lemma \ref{lemma-se-bottom}  $C_b$ is a trivial cylinder with one positive end $\beta_1^t$. This implies that $I = 1$, that is, only one component $C$ in the symplectic cobordism level $\overline{X}$. Then Remark \ref{rmk-lower-bound}, (3) and (4) in Lemma \ref{lemma-mix} imply that there are no $W^q$ components, that is, $C = u^1$. This curve $C$ can have positive ends as covers of $\alpha_1$ {\it and} covers of $\alpha_2$. Over the matching ends on (possibly covers of) $\alpha_2$, the constraint (\ref{ind=0}) implies that the only possibility for such curves in $SE(a,b)$ is from the last conclusion in (4) in Lemma \ref{lemma-mix}, that is, trivial cylinders with ends $\alpha_2$. Similarly, over the positive end on (possibly covers of) $\alpha_1$, the constraint (\ref{ind=0}) implies that the only possibility for such curves in $S\partial E(a,b)$ is from the last conclusion of (1) in Lemma \ref{lemma-mix}, that is, trivial cylinders with ends $\alpha_1$. In particular, as the unmatched positive ends of our limiting building include a single copy of the simple $\alpha_2$, this implies that the limit curve cannot be a multiple cover, that is, it must be somewhere injective as desired. 
\end{proof}

\subsection{Inductive gluing}\label{indgluing}  Notwithstanding the lengthy proof of Theorem \ref{thm-compact}, we also need to  confirm that the moduli space $\mathcal M_m$ are non-empty. In this subsection, we will use an inductive argument, essentially from McDuff's work \cite{McD18},  to obtain curves in the moduli space $\mathcal M^s_{J, \ep, S}(\alpha_1, …, \alpha_1, \alpha_2; \beta_1^t)$ defined by (\ref{dfn-moduli-2}), for a generic almost complex structure $J$ on the symplectic cobordism $\overline{X}$. 

\begin{lemma} \label{lemma-initial-2} There exist $\lambda>0$, $\ep>0$ and $\ep'>0$ sufficiently small such that $\lambda E(1, k+1+\ep')$ symplectically embeds into $E(a, ka+\ep)$ by inclusion. Moreover, for its associated symplectic cobordism $\overline{X}$ with a generic almost complex structure $J$ with cylindrical ends, there exists a rigid somewhere injective $J$-holomorphic curve with genus $0$, and with positive end $\alpha_2$ and negative end $\beta_1^{k+1}$. 
\end{lemma}

\begin{proof} Choose $\lambda = \frac{ak}{k+1}$, $\ep>0$ is arbitrarily small, and $\ep' = \frac{k+1}{2ak} \ep$. Then 
\[ \lambda = \frac{ak}{k+1} < a \,\,\,\,\,\mbox{and}\,\,\,\,\, \lambda (k+1+\ep') = ka + \frac{\ep}{2} < ka + \ep. \]
Hence, $\lambda E(1, k+1+\ep') \subset E(a, ka+\ep)$. Then by the first half of Example \ref{ex-1}, there exists a possibly broken $J$-holomorphic current $C$ with positive end $\alpha_2$ and negative end $\beta_1^{k+1}$ (as an orbit set). Meanwhile, its action difference 
\[ \Delta_{\mathcal A} = (ka + \ep) - (k+1)\lambda = (ka + \ep) - ka = \ep \]
which can be arbitrarily small. There is a lower bound on the action of any curves in the symplectization layers which do not cover trivial cylinders. Hence, there are no nontrivial curves in $S \partial E(a, ka+\ep)$, and there is a single curve in the cobordism with a single positive end covering $\alpha_2$. In particular, $C$ is somewhere injective, and Example \ref{ex-1} implies that $C$ is an embedded genus $0$ cylinder between $\alpha_2$ and $\beta_1^{k+1}$. By (\ref{dfn-f-ind}), 
\begin{align*}
{\rm ind}(C) & = ((2 + 2k+1) - \left(2(k+1) + 2 \floor*{\frac{k+1}{k+1+\ep'}} +1\right) \\ 
& = 2(k+1) - 2(k+1) = 0.
\end{align*}
Therefore, $C$ is rigid. Thus we complete the proof. 
\end{proof}

Note that curves provided by Lemma \ref{lemma-initial-2} do not have their negative ends lying on our required ellipsoid $E(\ep, \ep S)$. To resolve this, we will use a trick - glue them with cylinders obtained in \cite{HK18-2}. Recall that, under certain assumptions, Theorem 2 in \cite{HK18-2} says that  there exists a rigid cylinder in the symplectic cobordism between two ellipsoids, from a cover of the {\it short orbit} of an ellipsoid to a cover of  the {\it short orbit} of the other ellipsoid. By gluing these cylinders, we can correct the negative ends of the curves from Lemma \ref{lemma-initial-2} to lie on the desired $E(\ep, \ep S)$. Moreover, by gluing cylinders from \cite{HK18-2}, we will be able to get different combinations of $\alpha_1$ and $\alpha_2$. Here is the result which serves as the initial step. Recall that we always assume that $b = ka + \delta$ for arbitrarily small $\delta>0$ such that $b \notin \Q$.

\begin{prop} \label{prop-gluing} There exist an $\ep>0$ and an irrational $S \in (k+3, k+4)$ with $E(\ep, \ep S) \subset E(a, ka+\delta)$ such that for a generic almost complex structure $J$ with cylindrical ends on its associated symplectic cobordism, the moduli space 
\[ \mathcal M^s_{J, \ep, S}(\alpha_1, \alpha_2; \beta_1^{k+3}) \neq \emptyset. \]
\end{prop}

\begin{proof} First, we will provide a curve in the moduli space $\mathcal M^s_{J, \ep_0, S_0}(\alpha_2; \beta_1^{k+1})$ for some irrational number $S_0 \in (k+1, k+2)$ and $\ep_0$ such that $E(\ep_0, \ep_0 S_0) \subset E(a, ka + \delta)$. Fix any irrational $S_0 \in  (k+1, k+2)$ and a sufficiently small $\ep_0>0$ such that $E(\ep_0, \ep_0 S_0) \subset \lambda E(1, k+1+\ep')$ where $\lambda$ and $\ep'$ are taken as in Lemma \ref{lemma-initial-2}. In particular,  
\begin{equation} \label{S}
\ep_0 S_0 \leq \lambda(k+1 + \ep'), \,\,\,\,\mbox{that is}, \,\,\,\, \ep_0\leq \frac{2ka}{2S_0+1}.
\end{equation}
To remove the ambiguity of the notation for Reeb orbits, denote by $\gamma_1$ and $\gamma_2$ the primitive Reeb orbits of $\lambda E(1, k+1+\ep')$ with action $\lambda$ and $\lambda (k+1+\ep')$ respectively. For a generic almost complex structure $J$ of the resulting symplectic cobordism from $\lambda E(1, k+1+\ep')$ to $E(\ep_0, \ep_0 S_0)$, Theorem 2 in \cite{HK18-2} shows that there exists a rigid somewhere injective cylinder 
\begin{equation} \label{hk-1}
\mbox{$C_{\rm HK}$ from $\gamma_1^{k+1}$ (positive end) to $\beta_1^{k+1}$ (negative end)}.
\end{equation}
Indeed, the ellipsoids $E(\ep_0,\ep_0 S_0) \subset \lambda E(1, k+1+\ep')$ is a nested pair, and $k+1 < S = \frac{\ep S}{\ep}$.
Now, by gluing the curve $C'$ provided by Lemma \ref{lemma-initial-2} and $C_{\rm HK}$ in (\ref{hk-1}) along the matching end $\gamma_1^{k+1}$, we obtain a somewhere injective cylinder $C_0$ with Fredholm index ${\rm ind}(C_0) = 0 + 0 - 0 = 0$ by (\ref{f-ind-glue}). In other words, $C_0$ is rigid and $C_0 \in \mathcal M^s_{J, \ep_0, S_0}(\alpha_2; \beta_1^{k+1})$.

Next, by Theorem 2 in \cite{HK18-2} again, there exists a rigid somewhere injective cylinder 
\begin{equation} \label{hk-2}
\mbox{$C_{\rm HK}'$ from $\alpha_1$ (positive end) to $\beta_1$ (negative end)}
\end{equation}
in the symplectic cobordism from the ellipsoid embedding $E(\ep_0, \ep_0 S_0) \subset E(a,ka+\delta)$. This exists for any compatible almost-complex structure, in particular the same almost-complex structure for which $C_0$ is holomorphic.

In the symplectization $S\partial E(\ep_0, \ep_0 S_0)$, due to our hypothesis $k+1< S_0< k+2$, we have 
\[ \ceil*{\frac{1}{S_0}} + \ceil*{\frac{k+1}{S_0}}  = \ceil*{\frac{k+2}{S_0}} = 2. \]
Lemma 2.1 in \cite{McD18} checks there exists a curve $C_{\rm M}$ in the symplectization $S \partial E(\ep_0, \ep_0 S_0)$ with two positive ends $\beta_1$ and $\beta_1^{k+1}$ respectively and one negative end $\beta_1^{k+2}$ of index $0$ and genus $0$. 

Next, choose any 
\[ S \in (k+3, k+4) \,\,\,\mbox{and}\,\,\, 0< \ep< \ep_0 \,\,\,\,\mbox{such that $E(\ep, \ep S) \subset E(\ep_0, \ep_0 S_0)$.}\]
Theorem 2 in \cite{HK18-2} provides a rigid somewhere injective cylinder $C''_{\rm HK}$ from $\beta_1^{k+2}$ of $\partial E(\ep_0, \ep_0 S_0)$ to $\beta_1^{k+3}$ of $\partial E(\ep, \ep S)$.

Finally, we can apply Proposition 2.2 in \cite{McD18} to produce the desired curve in $\mathcal M^s_{J, \ep, S}(\alpha_1, \alpha_2; \beta_1^{k+3})$ by an obstruction bundle gluing of the curve $C_0, C'_{\rm HK}$ with $C_{\rm M}$ and $C''_{\rm HK}$. 
Figure \ref{figure-initial-step} below illustrates this gluing process. 
\begin{figure}[h]
\includegraphics[scale=0.7]{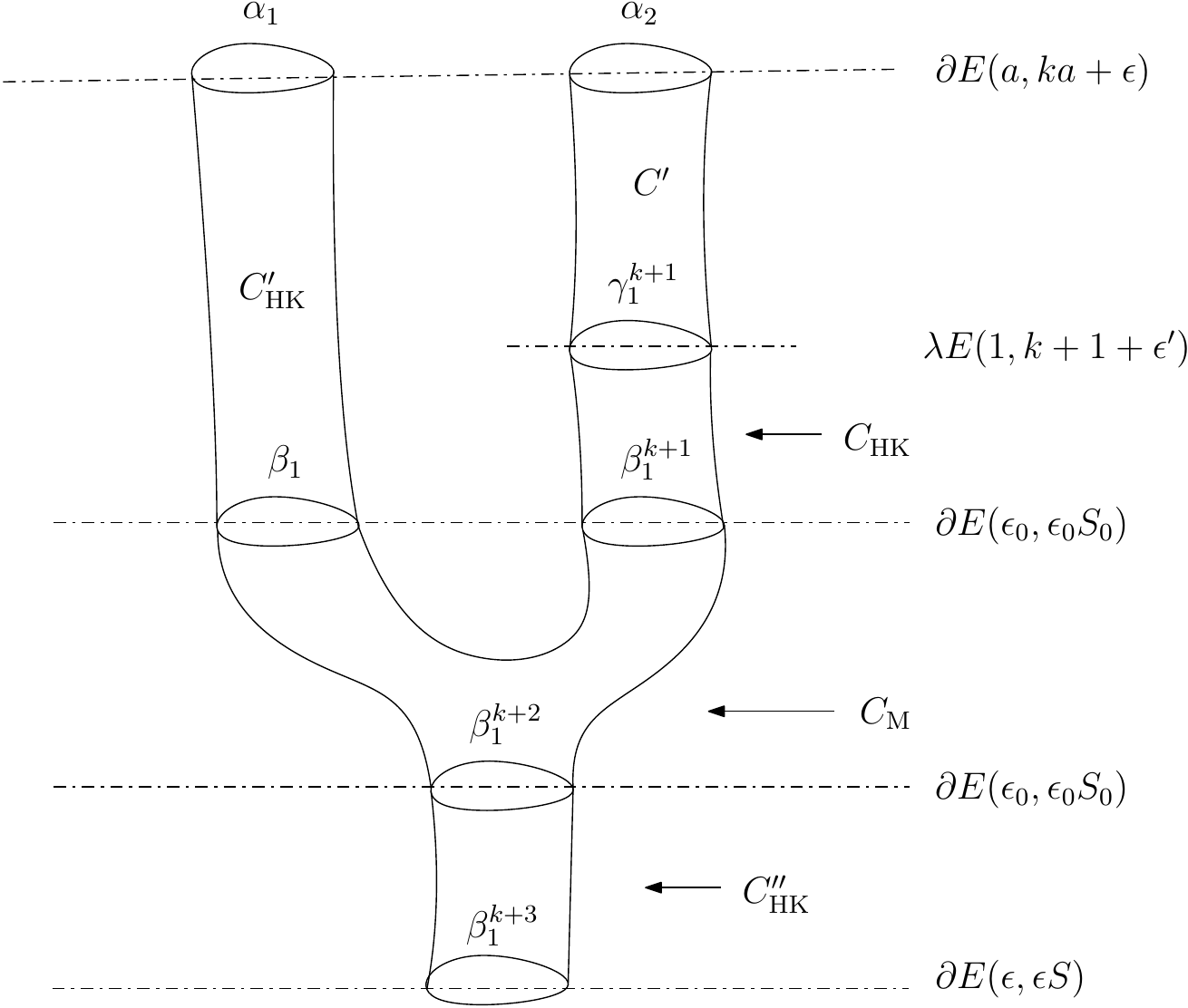}
\caption{Initial step of gluing}\label{figure-initial-step}
\end{figure}
\end{proof}

Now, here comes the main result in this section. 

\begin{theorem} \label{thm-gluing} For any $m \geq 1$, there exist $\ep_m>0$ and $S_m \notin \Q$ such that, for any generic almost complex structure $J$ of the associated symplectic cobordism, the moduli space defined in (\ref{dfn-moduli-2}) satisfies 
\[ \mathcal M^s_{J, \ep_m, S_m}(\underbrace{\alpha_1, …, \alpha_1}_{{\tiny \mbox{$m$-many}}}, \alpha_2; \beta_1^{k+1+2m}) \neq \emptyset\] 
where $\ep_m \to 0$ and $S_m \in (k+1+2m, k+2(m+1))$. In particular, $S_m \to \infty$. 
\end{theorem}

\begin{proof} Proposition \ref{prop-gluing} gives the initial step $m=1$. Here, we provide the proof of the inductive step. Denote by $t_m: = k+1 + 2m$. Suppose that the conclusion holds for $m>1$, that is, there exists a curve $C_m \in \mathcal M^s_{J, \ep_m, S_m}(\alpha_1, …, \alpha_1, \alpha_2; \beta_1^{t_m})$ with simple $\alpha_1$ in multiplicity $m$, where $S_m \in (t_m, t_m + 1)$. 

Meanwhile, Theorem 2 in \cite{HK18-2} provides a rigid somewhere injective cylinder $C_{\rm HK}$ from $\alpha_1$ of $\partial E(a,b)$ to $\beta_1$ of $\partial E(\ep_m, \ep_m S_m)$. As in Proposition \ref{prop-gluing} we may assume both $C_{\rm HK}$ and $C_m$ are holomorphic with respect to the same compatible almost complex structure.

In the symplectization $S\partial E(\ep_m, \ep_m S_m)$, due to the condition $t_m +1 > S_m > t_m$, we have 
\[ \ceil*{\frac{t_m}{S_m}} + \ceil*{\frac{1}{S_m}}  = \ceil*{\frac{t_m+1}{S_m}} = 2. \]
Hence, Lemma 2.1 in \cite{McD18} implies that there exists a rigid and genus $0$ curve $C_{\rm M}$ with two positive ends $\beta_1^{t_m}$ and $\beta_1$ respectively and one negative end $\beta_1^{t_m+1}$.

Choose an irrational $S_{m+1}$ and small $\ep_{m+1}>0$ such that 
\[ t_{m+1} > S_{m+1}> t_m + 1 \,\,\,\,\,\mbox{and}\,\,\,\,\, E(\ep_{m+1}, \ep_{m+1} S_{m+1}) \subset E(\ep_m, \ep_m S_m). \]
In particular, shrink $\ep_{m+1}$ if necessary such that $\ep_{m+1} < \ep_{m}$. Theorem 2 in \cite{HK18-2} now provides a rigid somewhere injective cylinder $C'_{\rm HK}$ from $\beta_1^{t_m+1}$ in $\partial E(\ep_m, \ep_m S_m)$ to $\beta_1^{t_{m+1}}$ in $\partial E(\ep_{m+1}, \ep_{m+1} S_{m+1})$.

Finally, Proposition 2.2 in \cite{McD18} gives a rigid curve $C_{m+1}$ by gluing $C_m, C_{\rm HK}$ with $C_{\rm M}$ and $C'_{\rm HK}$. 
This $C_{m+1}$ is the desired curve in $\mathcal M^s_{J, \ep_{m+1}, S_{m+1}}(\alpha_1, …, \alpha_1, \alpha_2; \beta_1^{t_{m+1}})$ with $\alpha_1$ in multiplicity $m+1$. Thus we complete the proof of the existence for all $m$ by induction. Figure \ref{figure-ind-2} illustrates this inductive process. 
\begin{figure}[h]
\includegraphics[scale=0.7]{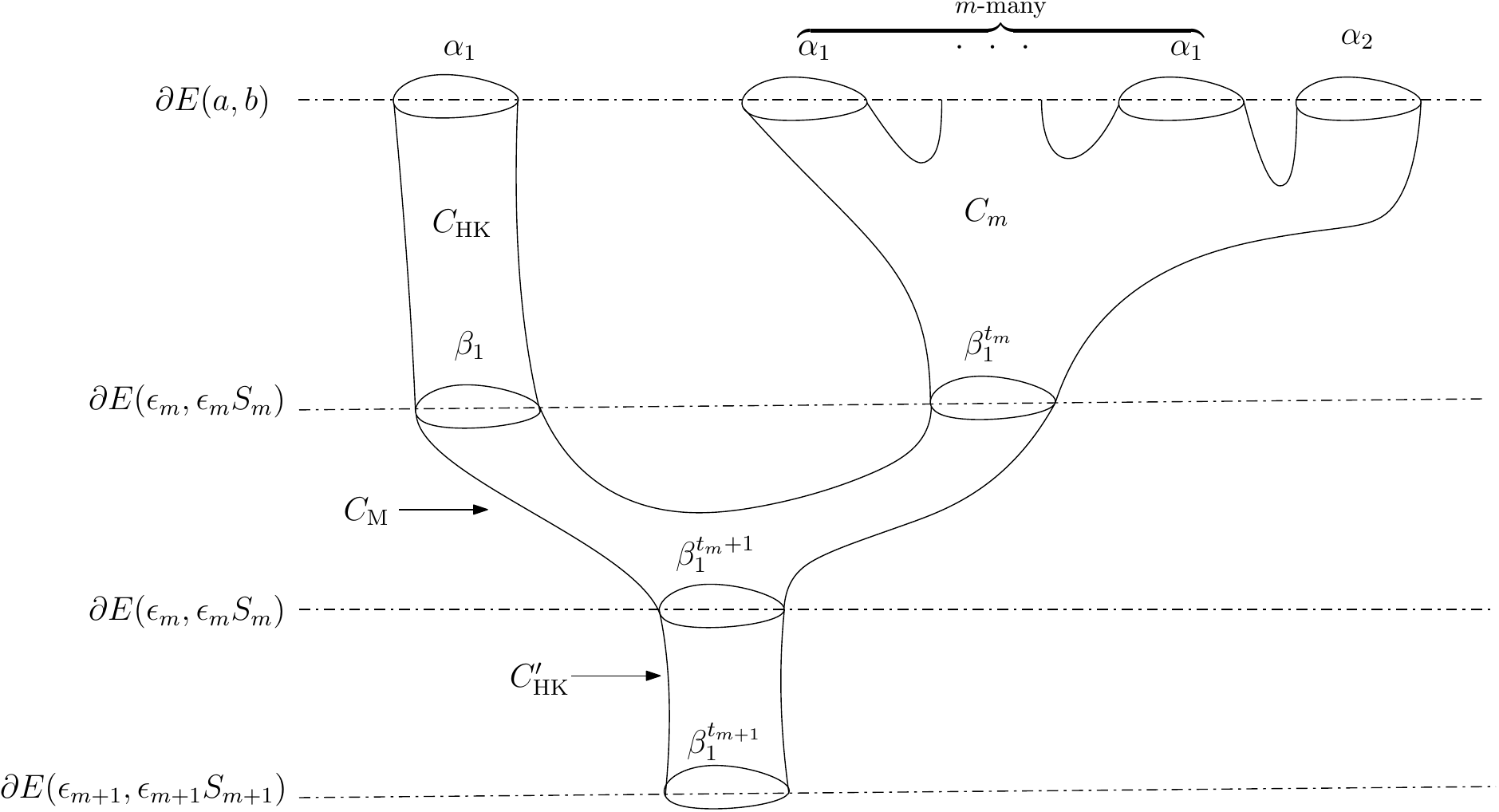}
\caption{Inductive process}\label{figure-ind-2}
\end{figure}

\end{proof} 

\section{Proof of Theorem \ref{thm-2}}\label{sec-proof}
In this section, we will provide the proof of Theorem \ref{thm-2}. Clearly, the proof will be divided into two parts. One is for the obstruction of embeddings, which will follow the main argument from \cite{HO19} but starting with the curves produced in Theorem \ref{thm-gluing}. The other is for the existence of embeddings, which comes from explicit symplectic folding constructions. 

\subsection{Obstruction} \label{ssec-obs} As explained in subsection \ref{ssec-outline}, we will start with 
\begin{equation} \label{curve-in-moduli}
C \in \mathcal M^s_{J, \ep_m, S_m}(\underbrace{\alpha_1, …, \alpha_1}_{{\tiny \mbox{$m$-many}}}, \alpha_2; \beta_1^{t_m})
\end{equation}
for $m \geq 1$ and $t_m = k+1 + 2m$. By Theorem \ref{thm-gluing} the moduli space is nonempty for particular $J$. Our curves are necessarily somewhere injective. Therefore, for generic $J$, we can apply \cite{wendlauto}, Theorem 1 and Corollary 3.17, to deduce automatic regularity (as our limiting Reeb orbits have odd Conley-Zehnder index the term $c_N$ in Theorem 1 is $-1$, and by Corollary 3.17 we may assume $Z(du) =0$). This implies that all curves in the moduli space have the same orientation, and it represents a nontrivial cobordism class. Together with the compactness result Theorem \ref{thm-compact} we deduce that the moduli space remains nonempty as we deform $J$. After a neck-stretching along the boundary $V \approx S^*_g \T^2$ there exists a limiting holomorphic building $C_{\rm lim}$. We will decompose $C_{\rm lim}$ as a union of components $F_0$ and $\{F_i\}_{i=1}^T$ as in Figure \ref{figure-climit}. To be coherent with the notation in subsection \ref{ssec-outline}, we will remove the subscript $m$ in (\ref{curve-in-moduli}), and moreover $d = t_m$. Recall that $S > d$. The following result is an analogue of Lemma 3.5 and Lemma 3.6 in \cite{HO19}. 

\begin{lemma} \label{lemma-index} With $F_0$ and $\{F_i\}_{i=1}^T$ defined as in subsection \ref{ssec-outline}, we have 
\[ {\rm ind}(F_0) = 0 \,\,\,\,\, \mbox{and}\,\,\,\,\, {\rm ind}(F_i) = 1\]
for any $i \in \{1, … , T\}$. Moreover, $T = d+1$. 
\end{lemma}

\begin{proof} The last conclusion directly comes from a Fredholm index computation and the first conclusion. In fact, by (\ref{dfn-f-ind}) and (\ref{c1-cz-torus}), 
\begin{align*}
0 = {\rm ind}(F_0) & = (T+1-2) + \frac{T}{2} +\frac{T}{2} - \left(2d + 2\floor*{\frac{d}{S}} +1 \right) \\
& = 2T - 2d -2. 
\end{align*}
Therefore, $T = d+1$. Moreover, it suffices to prove ${\rm ind}(F_0) \geq 0$ and ${\rm ind}(F_i) \geq 1$ for any $i \in\{1, …, T\}$. Indeed, since ${\rm ind}(C_{\rm lim}) = 0$, and as the Reeb orbits on $\partial V$ appear in $1$-dimensional families, by (\ref{f-ind-glue}),
\begin{align*}
0 = {\rm ind}(C_{\rm lim}) &= {\rm ind}(F_0) + \sum_{i=1}^T {\rm ind}(F_i) - T \geq 0 + (T-T) = 0. 
\end{align*}
This implies the desired conclusion. The first inequality ${\rm ind}(F_0) \ge 0$ has been proved in Lemma 3.6 in \cite{HO19}, so we are left to show that for any $i \geq \{1, …, T\}$, the Fredholm index ${\rm ind}(F_i) \geq 1$. However, observe that ${\rm ind}(F_i) \geq 0$ implies that ${\rm ind}(F_i) \geq 1$. Indeed, suppose that the positive ends of $F_i$ are $m_i$-many simple $\alpha_1$ and $\kappa$-many simple $\alpha_2$ where $\kappa = \{0,1\}$, while the negative end of $F_i$ is a closed Reeb orbit on $\partial V$ denoted by $\gamma_{(-k_i, -l_i)}$. Then by (\ref{dfn-f-ind}) and (\ref{c1-cz-torus-2}), ${\rm ind}(F_i) = 4m_i + \kappa (2k+4) + 2(k_i +l_i) -1$ which is always odd. Hence, it suffices to show ${\rm ind}(F_i) \geq 0$. 

\medskip

As shown in Figure \ref{figure-climit}, the component $F_i$ consists of curves in different layers. Let us view decompose $F_i$ as a union of $\{F^+_{ij}\}_{j=1}^{m_i^+}$, $\{F^0_{ij}\}_{j=1}^{m_i^0}$, and $\{F^-_{ij}\}_{j=1}^{m_i^-}$, where 
\begin{itemize}
\item{} $F^+_{ij}$ is a maximal connected union of curves in $F_i$ which is formed via gluing along matching ends such that after gluing it has unmatched negative ends only on $\partial E(a,b)$;
\item{} $F^0_{ij}$ is a single curve in $E(a,b) \backslash V$ that does not form a part of any $F^+_{ij}$-type component;
\item{} $F^-_{ij}$ is a component of the remainder, which is a union of curves at and below the layer $S\partial V$. 
\footnote{For component $F_T$ in Figure \ref{figure-climit}, it decomposes as a union of $\{F^+_{Tj}\}_{j=1}^4$, $\{F^0_{Tj}\}_{j=1}^2$, and $F^-_{T1}$, where from left to right,
\begin{itemize}
\item{} both $F^+_{T1}$ and $F^+_{T3}$ have 1 positive end and 1 negative end. $F^+_{T2}$ has 4 positive ends and 1 negative end. $F^+_{T4}$ has 2 positive ends and 1 negative end;
\item{} $F^0_{T1}$ has 1 positive end and 1 negative end. $F^0_{T2}$ has 3 positive ends and 2 negative ends;
\item{} $F^-_{T1}$ lies in $S\partial V$ and it has 2 positive ends and no negative end. 
\end{itemize}}
\end{itemize}
Observe that each $F^-_{ij}$ admits no negative ends. Assume that $F^-_{ij}$ has $s^-_{ij}$-many positive ends, then b) in Proposition 3.4 in \cite{HO19} implies that ${\rm ind}(F^-_{ij}) = 2s^-_{ij} -2 \geq s^-_{ij}$, since there is no finite energy plane in $V \approx T^*\mathbb T^2$ and thus $s_{ij}^- \geq 2$. As a result, 
\begin{align}\label{upper-level}
{\rm ind}(F_i) & = \sum_{j=1}^{m_i^+} {\rm ind}(F^+_{ij}) + \sum_{j=1}^{m_i^0} {\rm ind}(F^0_{ij}) + \sum_{j=1}^{m_i^-} {\rm ind}(F^-_{ij}) - \sum_{j=1}^{m_i^-} s^-_{ij} \nonumber\\
& \geq   \sum_{j=1}^{m_i^+} {\rm ind}(F^+_{ij}) + \sum_{j=1}^{m_i^0} {\rm ind}(F^0_{ij}). 
\end{align}
Therefore, it suffices to prove that (\ref{upper-level}) is non-negative. The following claim is useful. 

\begin{claim} \label{claim-one-neg-end} Each $F^+_{ij}$ has only one negative unmatched end. \end{claim}

\begin{proof} [Proof of Claim \ref{claim-one-neg-end}] By maximality of the components $F^+_{ij}$, each unmatched negative end will eventually, via components in lower layers, connect to the curve $F^0_{i j_0}$ with negative end on $\gamma_i$. As our curves have genus $0$ this can be true only for a single negative end. 
\end{proof}

In order to obtain the desired non-negativity of (\ref{upper-level}), we will carry out some analysis on the Fredholm indices of components $F_{ij}^+$ and $F^0_{ij}$. The general scheme is that a Fredholm index $F^0_{ij}$ could be negative due to the possibility of multiple coverings, but due to cancelling contributions the sum is always nonnegative.

Denote $u: = F_{ij}^0$ for some $i$,$j$. If $u$ is somewhere injective, then ${\rm ind}(u) = {\rm ind}(F_{ij}^0) \geq 0$. 
Otherwise assume that $u$ is a multiple covering of $\tilde{u}$ with covering number $r$, where $\tilde{u}$ is somewhere injective. Assume $u$ has $s^+$-many positive ends and $s^-$-many negative ends, similarly $\tilde{u}$ has $\tilde{s}^+$-many positive ends and $\tilde{s}^-$-many negative ends. We let $m_i$, $n_i$, $\tilde{m}_i$, $\tilde{n}_i$ be the covering numbers at $\alpha_1$ and $\alpha_2$ for $u$ and $\tilde{u}$ respectively.

We may assume ${\rm ind}(\tilde{u}) \ge 0$ and comparing ${\rm ind}(u)$ and $r {\rm ind}(\tilde{u})$ we obtain
\begin{align*}
{\rm ind}(u) & \geq  r{\rm ind}(\tilde{u}) +  (2r-2) - (2(r \tilde{s}^+ - s^+) + (r \tilde{s}^- - s^-)) \\
& + 2\left(\sum \floor*{\frac{m_i}{k+\ep}} - r \sum \floor*{\frac{\tilde{m}_i}{k+\ep}}\right) \\
& + 2\left(\sum \floor{n_i(k+\ep)} - r \sum \floor{\tilde{n}_i(k+\ep)}\right).
\end{align*}
Now, we apply the Riemann-Hurwitz formula.
Explicitly, view $u$ as a curve with domain a punctured $S^2$, and denote by $e_p$ the ramification index at point $p \in S^2$, then
\begin{align*}
2r - 2 & =  \sum_{p \in S^2} (e_p - 1) \\
& \geq \sum_{p \in \Gamma} (e_p-1) = \sum_{p \in \Gamma} (c_p - 1) = r(\tilde{s}^+ + \tilde{s}^-) - (s^+ + s^-)
\end{align*}
where $\Gamma$ is the set of singularity points on $S^2$ which corresponds to the asymptotic ends of $u$ (thus its cardinality is $s^+ + s^-$) and $c_p$ is the covering degree at the end corresponding to $p$. The sum of the $c_i$ covering each particular end is exactly $r$. It will be helpful to write $r \tilde{s}^+ - s^+ = \sum_{\Gamma^+} (c_i-1)$, where the $c_i$ are the covering degrees of the ends of $u$ over the positive ends of $\tilde{u}$. 

With this our estimate becomes 
\begin{align}\label{indexbound}
{\rm ind}(u) & \geq  r{\rm ind}(\tilde{u})  - (r \tilde{s}^+ - s^+) \nonumber \\
& + 2\left(\sum \floor*{\frac{m_i}{k+\ep}} - r \sum \floor*{\frac{\tilde{m}_i}{k+\ep}}\right) \nonumber \\
& + 2\left(\sum \floor{n_i(k+\ep)} - r \sum \floor{\tilde{n}_i(k+\ep)}\right).
\end{align}

Next, we note that by Claim \ref{claim-one-neg-end} each positive end of $u$ is matched to the single negative end of a component $F^+_{ij}$. These components have positive ends which are either simple covers of $\alpha_1$ or a simple cover of $\alpha_2$, and there can be at most one end (over all $F^+_{ij}$ matching with ends of $u$) asymptotic to $\alpha_2$.

By Lemma \ref{lemma-mix}, the $F^+_{ij}$ are either trivial cylinders or else have strictly positive index.
We will consider contributions to our formula from the ends of $u$ and the $F^+_{ij}$ corresponding to each positive end of $\tilde{u}$. 
We claim that the contribution to \eqref{upper-level} from the ends of a $u$ covering each end of $\tilde{u}$ in our estimate \eqref{indexbound} for ${\rm ind}(u)$, plus the associated component $F^+_{ij}$, is nonnegative unless the end of $\tilde{u}$ happens to be asymptotic to $\alpha_1$ and the corresponding $F^+_{ij}$ has a positive end on $\alpha_2$.

In the first case suppose that the end of $\tilde{u}$ is asymptotic to $\alpha^{\tilde{m}}_1$ and none of the covering ends of $u$ match to $F^+_{ij}$ asymptotic to $\alpha_2$. Using Lemma \ref{lemma-mix} (1), to check that the contribution from the covering numbers plus the corresponding $F^+_{ij}$ is nonnegative it suffices to show
$$\sum (2m_i - 2 -c_i + 1) \ge r \floor*{\frac{\tilde{m}}{k+\ep}}.$$
But $\sum m_i = r\tilde{m}$ and $\sum c_i  = r$, so the left hand side is just $$r(2\tilde{m} -1) - s \ge 2r(\tilde{m} -1)$$where $s \le r$ is the number of covering ends, that is, the number of terms in the sum. Hence, we get the inequality as required (recall that $k \ge 2$).

In the second case suppose that the end of $\tilde{u}$ is asymptotic to $\alpha^{\tilde{n}}_2$ and again none of the covering ends of $u$ match to $F^+_{ij}$ asymptotic to $\alpha_2$. Now by Lemma \ref{lemma-mix} (3) for nonnegativity it suffices to show
$$\sum (2n_i(2k-1) +2 -c_i + 1) \ge r \floor{\tilde{n}(k+\ep)}.$$
Now $\sum n_i = r\tilde{n}$, so the left hand side is just $$r(2\tilde{n}(2k-1) -1) + 3s \ge r(3\tilde{n}k -1).$$ Hence, we get the inequality as required. 

In the third case suppose that the end of $\tilde{u}$ is asymptotic to $\alpha^{\tilde{m}}_2$ and one end of one of the $F^+_{ij}$ is asymptotic to $\alpha_2$, say the end corresponding to $n_1$. Using Lemma \ref{lemma-mix} (4) we would like to show
$$(4k-2)n_1 -2k  + 2 -c_1 + 1 +
\sum_2^s (2n_i(2k-1) +2 -c_i + 1) \ge r \floor{\tilde{n} (k+\ep)}.$$
The left hand side is $2(2k-1)\tilde{n}r - r - 2k + 3s$ and we aim to show that $2(2k-1)\tilde{n}r - r - 2k + 3s \geq k \tilde{n} r$. In fact, since $r \geq s$ and $k \geq 2$, we have
\begin{align*}
2(2k-1)\tilde{n}r - r - 2k + 3s - k \tilde{n} r & = ((3k-2) \tilde{n}-1) r - 2k + 3s\\
& \geq (2k\tilde{n} -1)s + 3s - 2k \\
& = (2k\tilde{n} +2)s - 2k \geq 2k(\tilde{n}s -1) + 2s \geq 0
\end{align*}
since $s \geq 1$. Thus we get the required inequality.

Finally suppose that the end of $\tilde{u}$ is asymptotic to $\alpha^{\tilde{m}}_1$ and one end of one of the $F^+_{ij}$ is asymptotic to $\alpha_2$, say the end corresponding to $m_1$. Now from Lemma \ref{lemma-mix} (2) the contribution from ends of $u$ and the $F^+_{ij}$ associated to this end is bounded below by
$$2m_1 - 2k + 2 -c_1 + 1 +
\sum_2^s (2m_i - 2 -c_i + 1) - r \floor*{\frac{\tilde{m}}{k+\ep}}.$$
This is equivalent to
\begin{equation} \label{est-final0}
2r\tilde{m} - r - s - 2k + 4  - r \floor*{\frac{\tilde{m}}{k+\ep}}.
\end{equation}

As $m_1 \ge k+1$, we have $r\tilde{m} = \sum m_i \ge (k+1) + (s -1)\tilde{m}$ so the contribution is at least
\begin{equation} \label{est-final}
2(k+1 + (s-1)\tilde{m}) - r - s -2k + 4 -  r \floor*{\frac{\tilde{m}}{k+\ep}}.
\end{equation}
In the case when $\tilde{m} \le k$, we have (\ref{est-final}) bounded below by $$2(k+s) - r - s -2k + 4 = 4 - (r-s)$$ and in the case when $\tilde{m} \ge k+1$,
\begin{align*}
(\ref{est-final0}) & \geq (k+1 + (k+1)(s-1))\left(2 - \frac{1}{k}\right) - r - s - 2k + 4 \\
& \ge 2k + 2 - 1 - \frac{1}{k} + \frac{9}{2}(s-1) - r - s - 2k + 4\\
& \ge \frac{5}{2} s - (r-s).
\end{align*}
In any case the contribution is bounded below by $2 - (r-s)$.

In summary then, the only negative contributions to the sum \eqref{upper-level}, when we combine the $F^0_{ij}$ with matching $F^+_{ij}$, can come from a curve $F^0_{ij}$ with a positive end asymptotic to $\alpha_1$ and where the corresponding $F^+_{ij}$ has a positive end on $\alpha_2$. There exists at most one such $u = F^0_{ij}$. We focus on this curve and components of the complement. 

We note that if $\tilde{u}$ has a single negative end then it has an odd index and so the term $r{\rm ind}(\tilde{u}) \ge r$ which cancels the negative contribution to \eqref{indexbound}. Alternatively $\tilde{u}$ has two distinct negative ends, and we can suppose these are covered by $s_1$ and $s_2$ negative ends respectively of our curve $u= F^0_{ij}$. One of these $s_1 + s_2$ ends may be matched with $F_0$; the others match with curves in $F^-_{ij}$ and the complement of $u$ and these $F^-_{ij}$ in our building has components which are unions of curves of type $F^0_{ij}$ and $F^{\pm}_{ij}$. These components have nonnegative index, since they do not contain $u$, and if they are not matched to $F_0$ then they have a single unmatched negative end (which implies the index is odd) and so in fact have index at least $1$. At most one component has a curve matched to $F_0$, so we can add $s_1 + s_2 -1$ to our estimate for ${\rm ind}(u)$ in the sum \eqref{upper-level}.

Write $s = s_0$ (the number of ends of $u$ covering the distinguished positive end of $\tilde{u}$), and $s_i = r - (r-s_i)$ for $i=1,2$. Then adding the new terms our lower bound for \eqref{upper-level} becomes $2 + 2r -1 - \sum_0^2 (r-s_i)$. But by the Riemann-Hurwitz formula the sum is bounded above by $2r-2$, so the lower bound is at least $3$.

This gives desired non-negativity of (\ref{upper-level}) and thus we complete the proof. 
\end{proof}


\medskip

Now, we are ready to give the proof of ``only if'' part of Theorem \ref{thm-2}. 

\begin{proof} [Proof of ``only if'' part of Theorem \ref{thm-2}]
Suppose we have a $L(1,x) \hookrightarrow E(a,b)$, that is, a Lagrangian torus with an integral basis having area class $1$, $x$. Up to an arbitrarily small perturbation, assume that $b = ka + \delta \notin \Q$ for an arbitrarily small $\ep>0$. Then by embedding a thin ellipsoid $E(\ep, \ep S)$ inside a Weinstein neighborhood $V$ of the image of $L(1,x)$ inside $E(a,b)$, we obtain a symplectic cobordism $E(a,b) \backslash E(\ep, \ep S)$ for an almost complex structure $J$ with cylindrical ends. Due to Theorem \ref{thm-gluing}, there exist $J$ and curves $C \in \mathcal M^s_{J, \ep = \ep_m, S = S_m}(\alpha_1, …, \alpha_1,\alpha_2; \beta_1^{d})$ with $\alpha_1$ appearing multiplicity $m$ and $d= 2m+k+1$ for $m \geq 1$. By Theorem \ref{thm-compact} the curves persist as we carry out a neck-stretching process along the boundary of the Weinstein neighborhood viewed as a unit co-sphere bundle of $\T^2$. Denote the limit holomorphic building $C_{\rm lim}$. 
Decompose $C_{\rm lim}$ into $F_0$ and $\{F_i\}_{i=1}^T$ as shown in Figure \ref{figure-climit}, then Lemma \ref{lemma-index} shows that $T = d+1$, that is, $F_0$ has exactly $(d+1)$-many positive ends on $\partial V$, denoted by $\{\gamma_{(-k_i, -l_i)}\}_{i =1}^{d+1}$. Since these orbits together bound a contractible loop $\beta_1^d$ inside $U$, we know that 
\begin{equation} \label{top-1}
\sum_{i=1}^{d+1} k_i  = \sum_{i=1}^{d+1} l_i = 0.
\end{equation}

Assume a component $F_i$ has $m_i$-many positive ends on $\alpha_1$ and $\kappa$ on $\alpha_2$ where $\kappa \in \{0,1\}$. Then 
\begin{equation}\label{area-Fi-2}
{\rm Area}(F_i) = m_i a + \kappa b  + k_i + l_i x.
\end{equation}
The Fredholm index formula like (\ref{ind-Fi-2}) and Lemma \ref{lemma-index} imply that $2m_i + \kappa (k+2) + k_i + l_i =1$. By substituting this Fredholm index relation into (\ref{area-Fi-2}), we see
\begin{equation}\label{area-Fi-3}
m_i(a-2) + \kappa(b - (k+2)) \geq -l_i(x-1) -1.
\end{equation}

(1) In this case, as $x=1$ equation (\ref{area-Fi-3}) for a component $F_i$ gives
\[ m_i(a-2) + \kappa(b - (k+2)) \geq -l_i(x-1) -1=-1.\]
We can choose the $F_i$ with $\kappa =1$. If $a \geq 2$ then $b \geq 2 \geq \frac{1}{1 - 1/a}$ as required. Alternatively if $a<2$ then we see that $b - (k+2) \geq -1$ which is equivalent to the inequality as required.

\medskip

(2) Here $x \ge 2$ and we have two cases as follows. Note that since our ellipsoids are open, if there exists an embedded torus $L(1,2)$ then we can also find embedded tori $L(1,x)$ with $x>2$ by deforming in a Weinstein neighborhood. Hence, we will assume that $x >2$.

\medskip

\noindent Case I. Suppose that $l_i \geq 0$ for all $i \in \{1, …, d+1\}$ (thus by (\ref{top-1}), $l_i=0$ for all $i$), then by an appropriate choice of the metric on $V$ as in Section 2 in \cite{HO19}, we know that 
\[ {\rm Area}(F_0) = \sum_{i=1}^{d+1} |k_i| \frac{\ep}{2} - d \ep \]
and, in particular, $\sum_{k_i>0} k_i \geq d$. For those $i \in \{1, ..., d+1\}$ such that $k_i >0$, by (\ref{dfn-f-ind}) and (\ref{c1-cz-torus-2}), depending on whether the positive ends of $F_i$ include $\alpha_2$ or not, the corresponding Fredholm index of $F_i$ is 
\begin{equation} \label{ind-Fi-2}
 {\rm ind}(F_i) = 4m_i + 2k + 2k_i + 3 \,\,\,\,\mbox{or}\,\,\,\,  {\rm ind}(F_i) = 4m_i + 2k_i - 1
 \end{equation}
where $m_i$ is the number of positive end of component $F_i$ on $\alpha_1$. Then in order to have ${\rm ind}(F_i)=1$, the only possibility is $m_i = 0$, the positive end of $F_i$ does not include $\alpha_2$, and $k_i = 1$. In other words, these $F_i$ do not have any positive ends on $\partial E(a,b)$, that is, they are just planes, and there must be $d$ of them. 
Hence, there is only one component, say $F_{d+1}$, with one negative end $\gamma_{(d, 0)}$ on $\partial V$ and $(m+1)$-many positive ends on $\partial E(a,b)$. The following Figure \ref{figure-limit-conf-1} illustrates this configuration. 
\begin{figure}[h]
\includegraphics[scale=0.75]{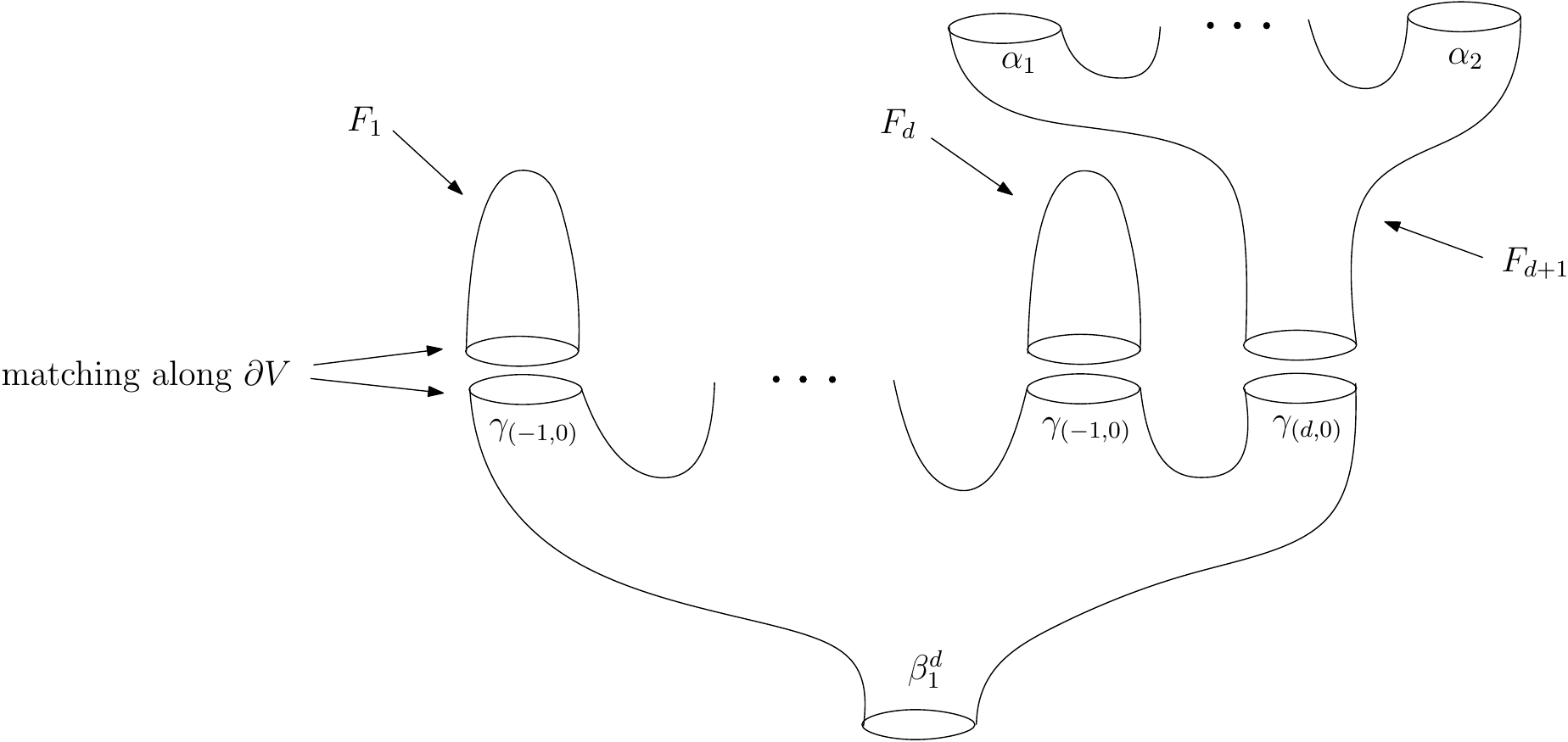}
\caption{One possible configuration of $C_{\rm lim}$}\label{figure-limit-conf-1}
\end{figure}
Since ${\rm Area}(F_{d+1}) >0$ and $d= 2m+k+1$, we have $ma + b \geq  2m + k+1$. This is equivalent to the following inequality
\[ a + \frac{b}{m} \geq 2 + \frac{k+1}{m}. \]
Repeating the argument with $m \to \infty$, we get $a \geq 2$. 

\medskip

\noindent Case II. Suppose there exists some $i \in \{1, …, d+1\}$ such that $\gamma_{(-k_i, -l_i)}$ has $l_i <0$. Since $l_i<0$, that is, $l_i \leq -1$, we get $-l_i(x-1) -1 \geq x-2$. Hence, by equation (\ref{area-Fi-3}) we have $m_i(a-2) + \kappa(b - (k+2)) \geq x-2$. Then, either $a > 2$ or $\kappa(b-(k+2)) \geq x-2$. The condition $x>2$ implies that $\kappa =1$ and thus 
\[ b - \left(\frac{b}{a}+2 \right) \geq x-2 \,\,\,\,\mbox{which implies that} \,\,\,\,b\left(1- \frac{1}{a} \right) \geq x.\]
Thus we complete the proof. \end{proof}

\subsection{Existence}\label{existence}
For the construction part of Theorem \ref{thm-2} we show the following, using coordinates $(z,w)$ on $\C^2$.

\begin{theorem}\label{constr} For all $x \ge 1$ there exists an embedded Lagrangian torus $L(1,x)$ in an arbitrary neighborhood of $E(2,4) \cap \{ \pi |w|^2 <2 \}$.
\end{theorem}

Note this is clear if $x \le 2$ by inclusion of the product, so the interest here is $x >2$.

\begin{cor} \label{cor-emb} Suppose $a > 2$ and $\frac{1}{a} + \frac{2}{b} \le 1$, then $L(1,x) \hookrightarrow E(a,b)$ for all $x \ge 1$.
\end{cor}

The second hypothesis is equivalent to saying $L(1,2)$ sits in $E(a,b)$ by inclusion.

\begin{proof} [Proof of Corollary \ref{cor-emb}] 
By the Theorem \ref{constr} it suffices to show $$E(2,4) \cap \{ \pi |w|^2 <2 \} \subset E(a,b).$$
For this, given $(z,w) \in E(2,4) \cap \{ \pi |w|^2 <2 \}$ we compute
$$ \frac{ \pi|z|^2}{a} + \frac{  \pi|w|^2}{b} \le \frac{ \pi|w|^2}{b} + \frac{ 2 ( 1 - \pi|w|^2/4)}{a}=\frac{2}{a} +  \pi|w|^2\left(\frac{1}{b} - \frac{1}{2a}\right).$$
Since $a > 2$ and $\pi |w|^2 <2$ this is bounded by $1$ as required.
\end{proof}

\begin{proof} [Proof of Theorem \ref{constr}] We adjust the construction in \cite{HO19}. This gives an $L(1,x) \hookrightarrow D(2) \times \Omega$ (or a neighborhood of this) where $\Omega$ is a subset of the $w$ plane in coordinates $(s,t)$, roughly $[0,2] \times [0,1] \setminus \{(1,t) | \ep \le t  \le 1\}$, see Figure 1 in \cite{HO19}.

The $L(1,x)$ is formed by first taking the product of an immersed loop $\gamma$ in the $w$ plane, which winds a large number of times around $\partial ([0,1] \times [0,1])$, with $\partial D(1)$ in the $z$ plane. The product gives an immersed Lagrangian which can be arranged to have the Maslov and area class of $L(1,x)$. All of the self intersections lie close to $\partial D(1) \times \{(1,0)\}$ and these are removed by applying a Hamiltonian diffeomorphism to the subsets $\partial D(1) \times H$ and $\partial D(1) \times V$ where $H$ is the uppermost branch of $\gamma$ close to $t=0$ and $V$ is the rightmost branch close to $s=1$. The diffeomorphisms are generated by functions $\chi(s)G(z)$ and $\chi(1-t)G(z)$ respectively, where $\chi$ is a function increasing from $0$ to $1$ on $[0,1]$ with slope bounded approximately by $1$ and $G$ displaces $D(1)$ inside $D(2)$ and has $0 \le G \le 1$.

The result of this is that points in $\partial D(1) \times H$ flow in the positive $t$ direction, and points in $\partial D(1) \times V$ flow in the positive $s$ direction, and we get an embedded Lagrangian in $D(2) \times \Omega$. Moreover, the fibers over a point $(s,t)$ in the $w$ plane lie in a disk $D(a)$ where $a = 1 + \chi(s) = 1+s$ when $s <1$ and $a=1+ \chi(t) = 2-t$ when $s>1$. This is because the $z$ component of the flow is $\chi X_G$.

We conclude by applying a symplectomorphism mapping $\Omega \to D(2)$. This can be chosen such that $\{s <1\}$ maps into the upper half $U$ of the disk and $\{s>1\}$ maps to the lower half $L$. Moreover, we can arrange that $\{[s,1] \times [0,1]\} \to D(2(1-s)) \cap U$ and $\{[1,2] \times [0,t]\} \to D(2t) \cap L$. We take the product of this map with the identity in the $z$ plane and compute the fibers over points $w \in D(2)$.

If $w \in U$ and $\pi|w|^2=\rho$ then $w$ lies in the image of a point $(s,t) \in \Omega$ with $s \le 1-\rho/2$ and thus the fiber lies in $D(2-\rho/2)$. Similarly, if $w \in L$ and $\pi|w|^2=\rho$ then $w$ lies in the image of a point $(s,t) \in \Omega$ with $t \ge \rho/2$ and thus the fiber lies in $D(2-\rho/2)$. Hence, the image lies close to $$\left\{ (z,w) \in \C^2 \bigg| \, \pi |z|^2 \le 2 - \frac{ \pi |w|^2}{2} \right\} = E(2,4).$$ Thus we complete the proof. \end{proof}

Proposition 6.3 in \cite{HO19} goes on to show that in fact the constructed tori $L(1,x)$ can be taken to be Hamiltonian isotopic to the product tori $L_H(1,x)$. This justifies the equality (\ref{ham=red}) in subsection \ref{hamshape}.

\section{Applications} \label{sec-app}

We start this section by proving the following symplectic invariant property of the reduced shape invariant ${\rm Sh}^+$ of star-shaped domains of $\C^2$. 

\begin{prop} \label{prop-ham-inv} Let $X, Y$ be a star-shaped domains of $\C^2$. If there exists a symplectic embedding $\phi: X \to Y$, then ${\rm Sh}^+(X)\subset {\rm Sh}^+(Y)$. In particular, ${\rm Sh}^+$ is a symplectic invariant of star-shaped domains of $\C^2$, i.e., ${\rm Sh}^+(X) = {\rm Sh}^+(Y)$ if $X$ and $Y$ are symplectomorphic.  \end{prop}

\begin{proof} The second conclusion immediately follows from the first conclusion, so it suffices to prove the first one. Let $\phi: X \to Y$ be a symplectic embedding, then Proposition A.1 in \cite{Sch-book} implies that $\phi$ can be extended to be a compactly support Hamiltonian diffeomorphism $\Phi$ on $\C^2$ such that $\Phi(X) \subset Y$. If $(w_1, w_2) \in {\rm Sh}^+(X)$, by definition (\ref{dfn-red-si}), there exists an embedded Lagrangian torus $L$ in $X$ such that $\Omega(e_1) = w_1$ and $\Omega(e_2) = w_2$, where $(e_1, e_2)$ is the ordered basis of $H_1(L; \Z)$ provided by Lemma \ref{lemma-1}. Then $\Phi(L)$ is an embedded Lagrangian torus in $Y$ and $\{\Phi_*(e_1), \Phi_*(e_2)\}$ is an integral basis of $H_1(\Phi(L); \Z)$. On the one hand, we have 
\begin{equation} \label{inv-maslov}
\mu(\Phi_*(e_1)) = \mu(\Phi_*(e_2)) = 2
\end{equation}
since $\Phi_*$ is isotopic to the identity. On the other hand, 
\begin{equation} \label{inv-areaclass}
\Omega(\Phi_*(e_1)) = \Omega(e_1) \,\,\,\,\mbox{and}\,\,\,\, \Omega(\Phi_*(e_2)) = \Omega(e_2). 
\end{equation}
In fact, assume $\Phi = (\Phi_t)_{t \in [0,1]}$, then 
\begin{align*}
\Phi^*\lambda_{\rm std} - \lambda_{\rm std}  = \int_0^1 \Phi_t^*(d(H_t + \iota_{X_t} \lambda_{\rm std})) dt = d \left(\int_0^1 \Phi_t^*(H_t + \iota_{X_t} \lambda_{\rm std}) dt \right). 
\end{align*}
In other words, set $F = \int_0^1 \phi_t^*(H_t + \iota_{X_t} \lambda_{\rm std}) dt$, then $\Phi^*\lambda_{\rm std} - \lambda_{\rm std} = dF$. Hence, for $e_1$ (and similarly to $e_2)$, we have 
\[ \Omega(\Phi_*(e_1)) - \Omega(e_1) = [\Phi^*\lambda_{\rm std} - \lambda_{\rm std}](e_1) = [dF](e_1) = 0. \]
Finally, (\ref{inv-maslov}) and (\ref{inv-areaclass}) imply that $(\Phi_*(e_1), \Phi_*(e_2))$ is the basis of $H_1(\Phi(L); \Z)$ provided by Lemma \ref{lemma-1}. Indeed, if $L$ is monotone, then $\Phi(L)$ is monotone; if $L$ is non-monotone, then $\Omega(\Phi_*(e_2)) = \Omega(e_2) \geq 2 \Omega(e_1) = 2 \Omega(\Phi_*(e_1))>0$. Therefore, $(w_1, w_2) \in {\rm Sh}^+(Y)$ and thus we complete the proof. \end{proof}

For any domain $X$ in $\C^2$, by Darboux's Theorem in symplectic geometry, its reduced shape invariant ${\rm Sh}^+(X)$ always contains a ball neighborhood of $0$ (without the point $0$) that restricted in the region $\{w_1>0, w_2 \geq 2 w_1\} \cup\{w_1>0, w_1 = w_2\}$. Then it makes sense to discuss the radial rescaling of ${\rm Sh}^+(X)$ with respect to $0$. In fact, it is readily to check that ${\rm Sh}^+(\lambda X) = \lambda {\rm Sh}^+(X)$. In this way, a quantitative improvement of Proposition \ref{prop-ham-inv} is a stability that relates a quantitative measurement between domains $X$ and $Y$ (say, symplectic Banach-Mazur distance in \cite{Ush19,SZ18}) and a 2-dimensional quantitative inclusions between the associated ${\rm Sh}^+$. Researches in this direction have been explored in the upcoming work \cite{HZip}. 

\begin{ex} \label{ex-hsi} This example collects the reduced shape invariants of several basic toric domains in $\C^2$, that is, symplectic cylinder $Z(R) := E(R, \infty)$, ball $B(R):= E(R, R)$, polydisk with $a \leq b$, and ellipsoid $E(a,b)$ with $\frac{b}{a} \in \N_{\geq 2}$. They are illustrated in the the shade dotted areas. In particular, for polydisk $P(a,b)$ and its reduced shape invariant ${\rm Sh}(P(a,b))$, whether the point $(a,b)$ lies above the line $2w_1 = w_2$ or below depends on whether $\frac{b}{a} \geq 2$ or not. These pictures are respectively due to the main result in \cite{Che96}, Theorem 1 in \cite{HO19},  Theorem 3 in \cite{HO19}, and Theorem \ref{thm-1} (or Theorem \ref{thm-2}) in this paper. 

\[
\includegraphics[width=6.3cm]{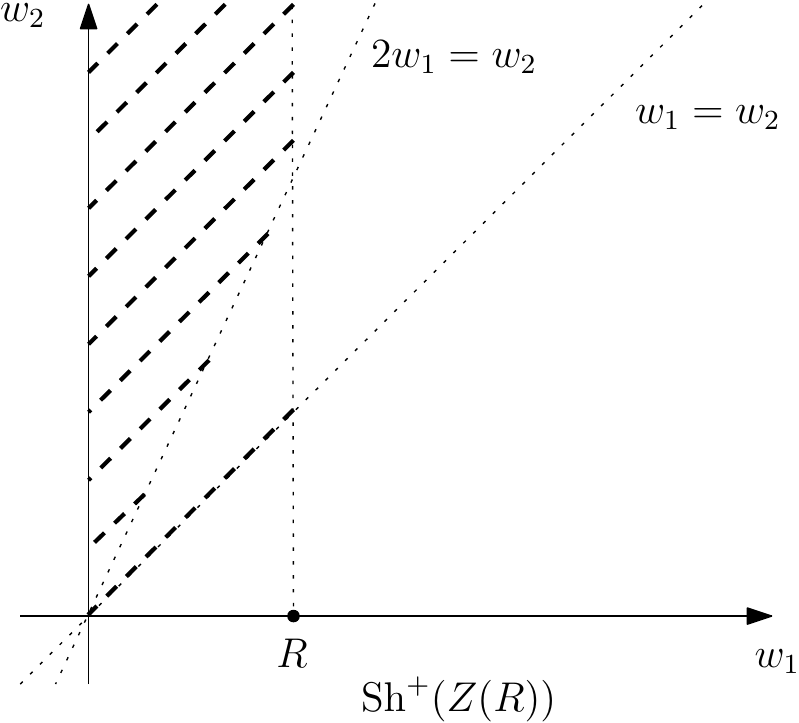} \,\,\,\,
\includegraphics[width=6.5cm]{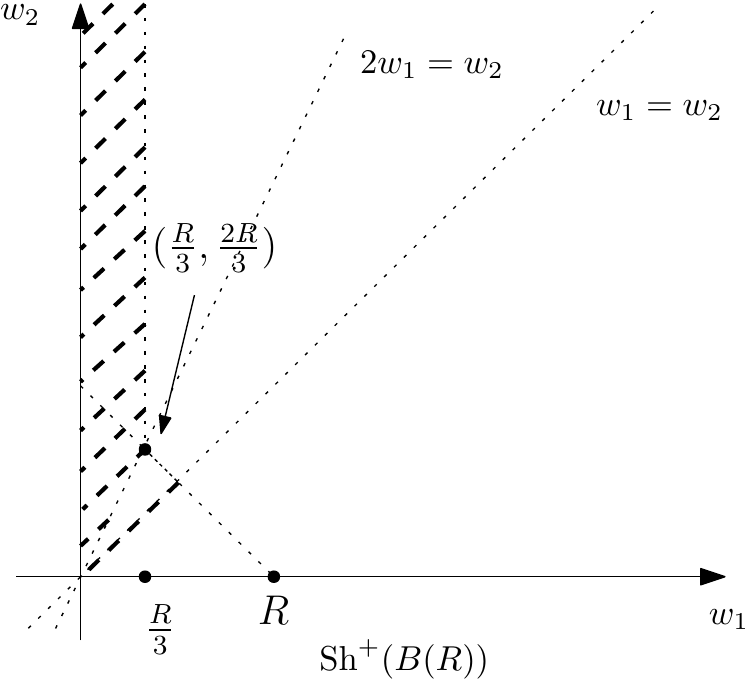}
\]
\[
\includegraphics[width=6.3cm]{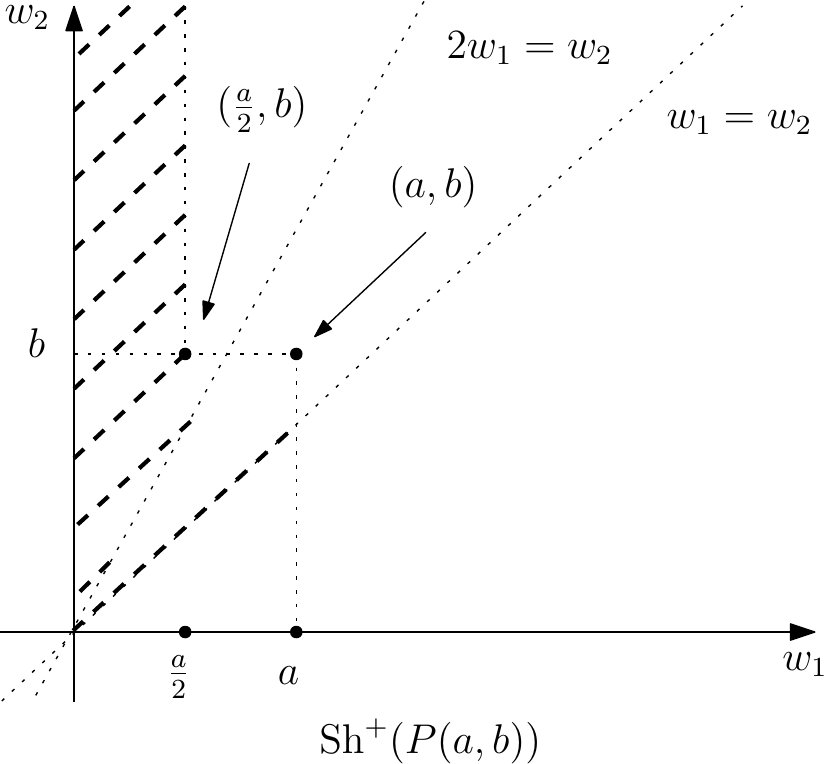} \,\,\,\,
\includegraphics[width=6.3cm]{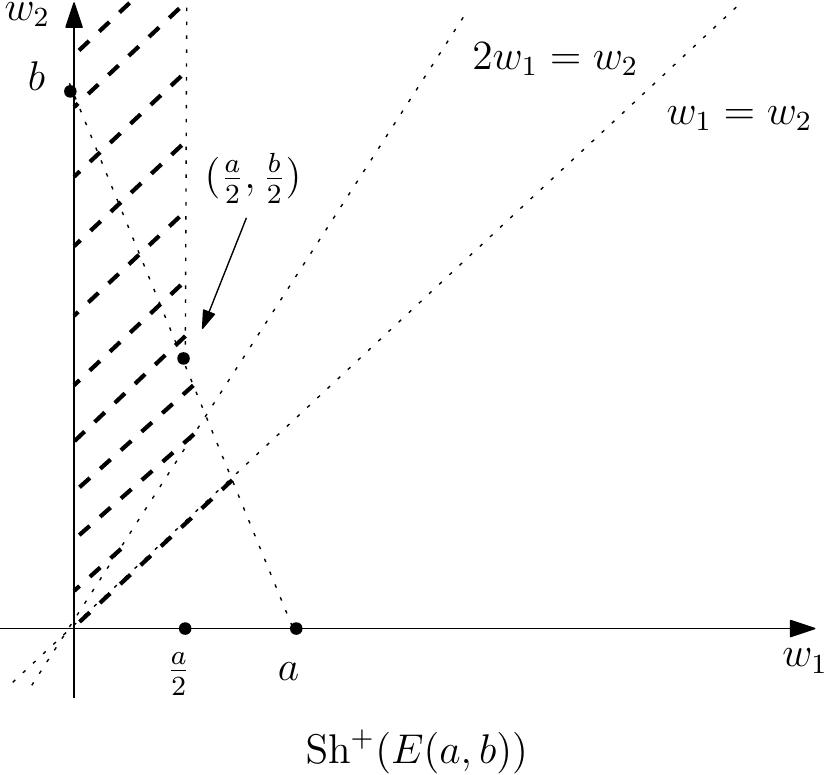}
\]
\end{ex}

Next, let us give the proof of Theorem \ref{thm-3}. 

\begin{proof} [Proof of Theorem \ref{thm-3}] (1) If $P(a,b) \hookrightarrow P(c,d)$, then Proposition \ref{prop-ham-inv} implies that 
\begin{equation} \label{hsi-inc-1}
{\rm Sh}^+(P(a,b)) \subset {\rm Sh}^+(P(c,d)).
\end{equation}
It suffices to consider the case when $\frac{b}{a} \geq 2$. Indeed, if $\frac{b}{a} <2$, then the turning point $(\frac{c}{2}, d)$ of ${\rm Sh}^+(P(c,d))$ must lie in ${\rm Sh}^+(P(a,b))$, which contradicts (\ref{hsi-inc-1}). For $\frac{b}{a} \geq 2$, the turning point $(a,b)$ of ${\rm Sh}^+(P(a,b))$ lies on and above the line $2 w_1 = w_2$. Then the desired conclusion can be obtained via a straightforward comparison between the reduced shape invariants ${\rm Sh}^+(P(a,b))$ and ${\rm Sh}^+(P(c,d))$ shown in Example \ref{ex-hsi}. Figure \ref{Figure-polydisk-polydisk} illustrates this comparison. 
\begin{figure}[h] 
\includegraphics[width=6.7cm]{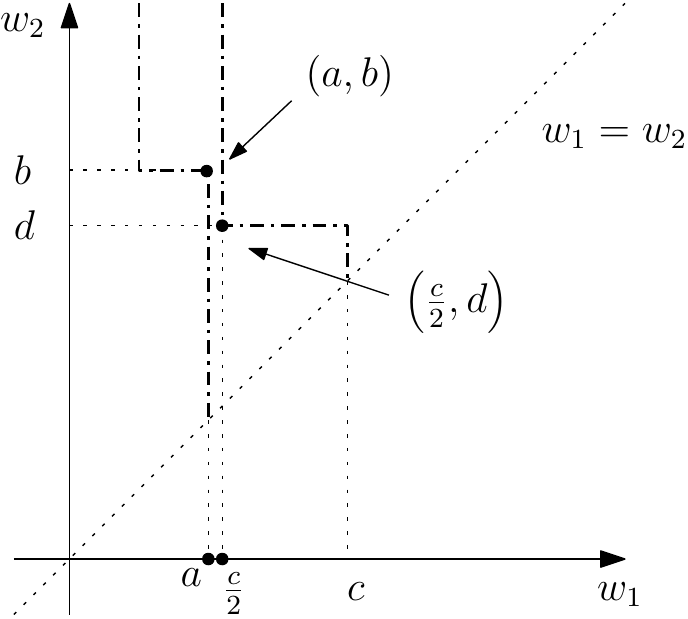} 
\caption{Obstruction of polydisk embedding via ${\rm Sh}^+$.}\label{Figure-polydisk-polydisk}
\end{figure}
In Figure \ref{Figure-polydisk-polydisk}, if $\frac{c}{a} <2$, that is, $\frac{c}{2} <a$, then there exists some point, say,
\[ \left(\frac{\frac{c}{2} + a}{2}, \frac{b+d}{2} \right) \in {\rm Sh}^+(P(a,b)) \backslash {\rm Sh}^+(P(c,d)). \]
Contradiction. Thus we get the conclusion.

\medskip

(2) If $a + b\leq bc$, it is easy to verify that $E(c,bc)$ embeds into $P(1,a)$ by the trivial inclusion. On the other hand, for ellipsoid $E(c, bc)$ with $b \in \N_{\geq 2}$, Theorem \ref{thm-1} applies. If $P(1,a) \hookrightarrow E(c, bc)$ with $1\leq c \leq 2$, then Proposition \ref{prop-ham-inv} implies that 
\begin{equation} \label{hsi-inc-2}
{\rm Sh}^+(P(1,a)) \subset {\rm Sh}^+(E(c,bc))).
\end{equation}
The desired conclusion comes from a comparison between the reduced shape invariants ${\rm Sh}^+(P(1,a))$ and ${\rm Sh}^+((E(c,bc))$ shown in Example \ref{ex-hsi}. Figure \ref{Figure-polydisk-ellipsoid} illustrates this comparison. 
\begin{figure}[h] 
\includegraphics[width=7.5cm]{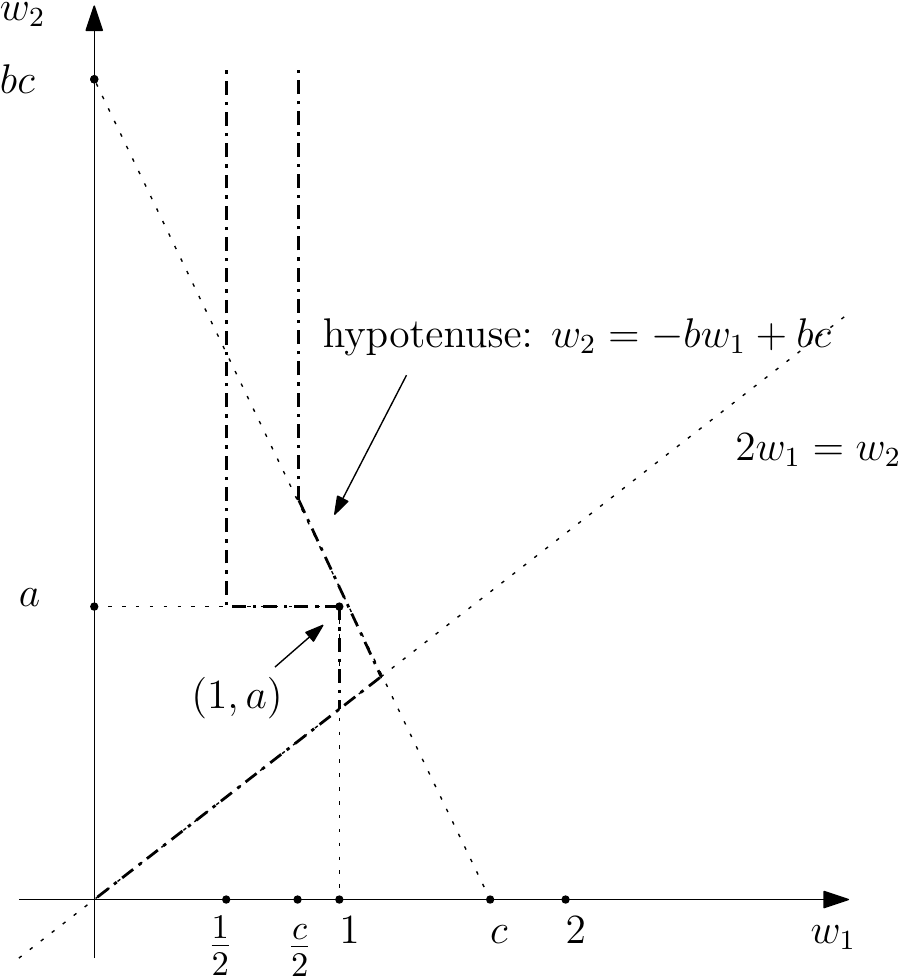} 
\caption{Obstruction of embedding from polydisk to ellipsoid via ${\rm Sh}^+$.}\label{Figure-polydisk-ellipsoid}
\end{figure}
In Figure \ref{Figure-polydisk-ellipsoid}, in order to have the resulting inclusion as in (\ref{hsi-inc-2}), the point $(1,a) \in {\rm Sh}^+(P(1,a))$ must lie {\it on or below} the hypotenuse $w_2 = -bw_1 + bc$. In other words, 
\[ a \leq -b + bc, \,\,\,\,\mbox{which is $a+b \leq bc$ as required}. \]
Thus we complete the proof. 
\end{proof}

\begin{remark} The Figure \ref{Figure-polydisk-ellipsoid}  shows that if $c \geq 2$, that is, $\frac{c}{2} \geq 1$, then the inclusion (\ref{hsi-inc-2}) always holds. This implies that comparing the reduced shape invariants will not result in any effective obstructions. \end{remark}

\section{Proof of Theorem \ref{conj-1}} \label{conjsec}

Here we describe how a similar argument to the proof of Theorem \ref{thm-2} can also be used to establish Theorem \ref{conj-1}, except that in this case we need new methods developed by Siegel in \cite{Sie20,Sieip} to show that the key moduli space is nonempty.

\medskip

Given Theorem \ref{thm-2} and the comments in section \ref{hamshape}, Theorem  \ref{conj-1} follows from the following.

\begin{theorem}\label{hamemb} Let $1<x<2$ and $k = \frac{b}{a} \in \N_{\geq 2}$. Then $$L(1,x) \hookrightarrow E(a,b)$$ if and only if $x < b(1-\frac{1}{a})$.
\end{theorem}

The theorem claims that inclusions give optimal embeddings, so it suffices to show that the condition is necessary. For this we will follow the scheme in section \ref{sec-proof}, but the key ingredient will be the following analogue of Theorem \ref{thm-gluing}.

Setting $d = (k+2)m -1$ (so that the Fredholm index is $0$) we define
\begin{equation} \label{dfn-moduli-3}
\mathcal N^s_{J, \ep, S}(\underbrace{\alpha_2, …, \alpha_2}_{{\tiny \mbox{$m$-many}}}; \beta_1^d) := {\small \left\{\begin{array}{l} \mbox{simple $J$-holomorphic curve $u$ with positive} \\ \mbox{ends $(\alpha_2, …, \alpha_2)$ and negative end $\beta_1^d$}  \end{array} \right\}}.
\end{equation}

\begin{theorem} \label{thm-gluing2} For any $m \geq 1$, there exist $\ep_m>0$ and $S_m \notin \Q$ such that, for any generic almost complex structure $J$ of the associated symplectic cobordism, the moduli space
\[ \mathcal N^s_{J, \ep_m, S}(\underbrace{\alpha_2, …, \alpha_2}_{{\tiny \mbox{$m$-many}}}; \beta_1^{k+1+2m})\]
defined in (\ref{dfn-moduli-3}) is nonempty whenever $S$ is very large.
\end{theorem}

Assuming this for now, we proceed to the proof of Theorem \ref{hamemb}. The arrangement is as in section \ref{ssec-outline}, that is, we have inclusions 
\[ E(\ep, \ep S) \subset V \subset E(a,b),\]
where $V$ is a perturbed unit disk bundle of the embedded Lagrangian torus, and we apply our neck stretching procedure to curves in $\mathcal N^s$ as the almost complex structure is deformed along $\partial V$.

The compactness result, Theorem \ref{thm-compact} holds equally well for the moduli space $\mathcal N^s$. For brevity we omit the proof, which follows the same strategy as Theorem \ref{thm-compact} even though it involves a different moduli space. We conclude the curves in $\mathcal N^s$ persist under deformations of the almost complex structure, and we obtain a nontrivial limit.

Lemma \ref{lemma-index} still applies in this case to describe the limit and we analyze the possibilities for the components $F_i$, $1 \le i \le T=d+1$ in the completion of $E(a,b) \setminus V$. Here $d= (k+2)m -1$ is the degree of the covering of the negative end of curves in $\mathcal N^s$, and hence also the negative end of the component $F_0$ inside $V$.
The component $F_0$ has exactly $T=(d+1)$-many positive ends on $\partial V$, denoted by $\{\gamma_{(-k_i, -l_i)}\}_{i =1}^{d+1}$. Since these orbits together bound a contractible loop $\beta_1^d$ inside $U$, we know that 
\begin{equation} \label{top-1}
\sum_{i=1}^{d+1} k_i  = \sum_{i=1}^{d+1} l_i = 0.
\end{equation}
Now, we have two cases as follows.  

\medskip

\noindent In the first case suppose that $l_i \geq 0$ for all $i \in \{1, …, d+1\}$ (thus by (\ref{top-1}), $l_i=0$ for all $i$), then by an appropriate choice of the metric on $V$ as in Section 2 in \cite{HO19}, we know that 
\[ {\rm Area}(F_0) = \sum_{i=1}^{d+1} |k_i| \frac{\ep}{2} - d \ep \]
and, in particular, $\sum_{k_i>0} k_i \geq d$. For $i \in \{1, ..., d+1\}$ , by (\ref{dfn-f-ind}) and (\ref{c1-cz-torus-2}), the corresponding Fredholm index of $F_i$ is 
\begin{equation} \label{ind-Fi}
 {\rm ind}(F_i) = m_i(4+2k) - 1 + 2k_i + 2l_i
 \end{equation}
where $m_i$ is the number of positive ends of component $F_i$. In the case when $k_i >0$, since $l_i=0$, we see that if $m_i$ is non-zero, then ${\rm ind}(F_i) \geq 4 - 1 =3$, so we cannot have ${\rm ind}(F_i) = 1$ as promised in Lemma \ref{lemma-index}. Therefore, for these $F_i$, we have $m_i=0$. In other words, these $F_i$ do not have any positive ends on $\partial E(a,b)$ and are planes. Moreover, for these $F_i$ to have ${\rm ind}(F_i) = 1$, we must have $k_i = 1$. Then there is only one remaining component, say $F_{d+1}$, with a negative end $\gamma_{(d, 0)}$ on $\partial V$ and $m$-many positive ends on $\partial E(a,b)$. The following Figure \ref{figure-limit-conf-2} illustrates this configuration (cf.~Figure \ref{figure-limit-conf-1}). 
\begin{figure}[h]
\includegraphics[scale=0.75]{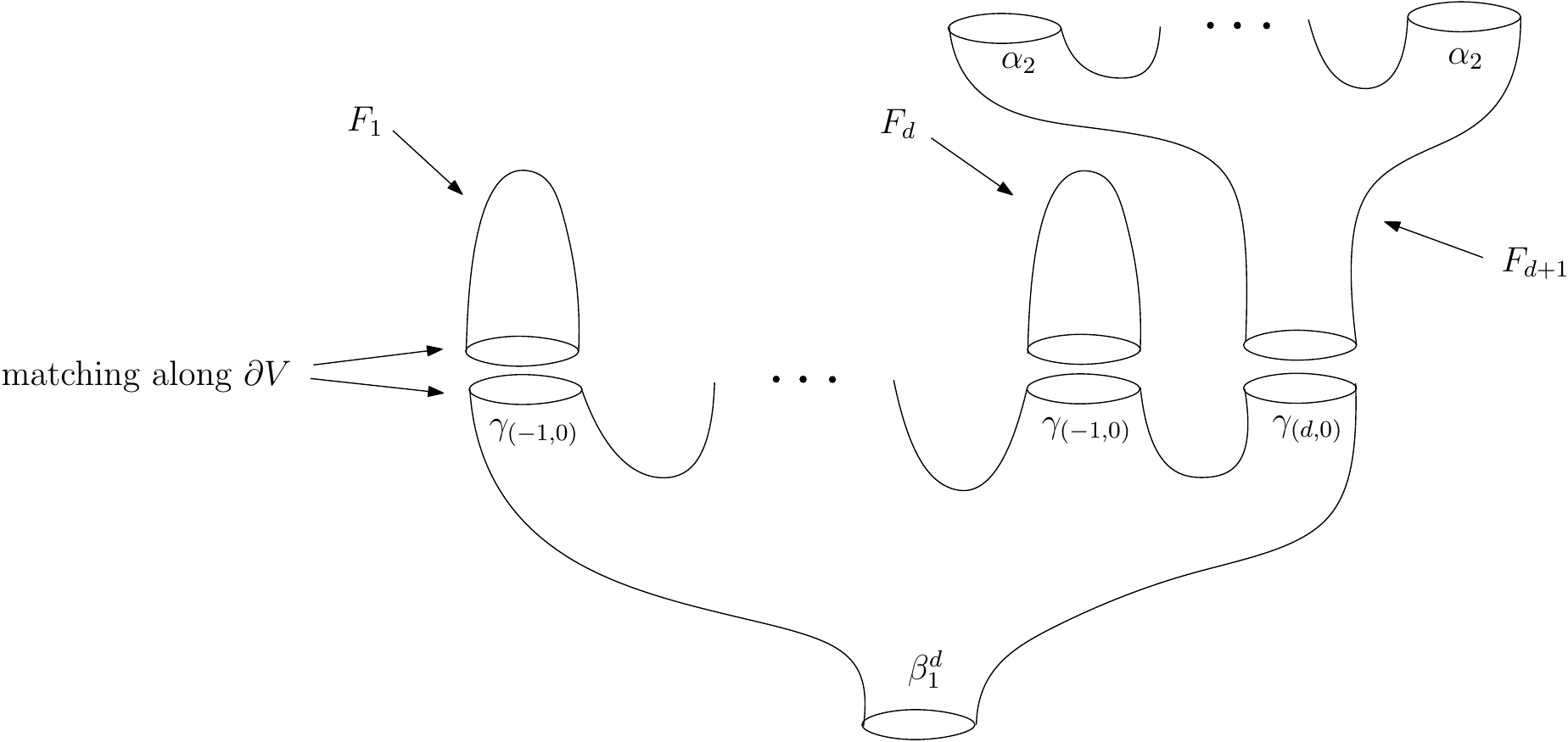}
\caption{One possible configuration of $C_{\rm lim}$}\label{figure-limit-conf-2}
\end{figure}

\noindent Since ${\rm Area}(F_d)\geq 0$, we have $mb \geq d$. Together with $d = (k+2)m-1$, we have 
\[ mb \geq m(k+2) - 1, \,\,\,\,\mbox{which is} \,\,\,\, b \geq k+2 - \frac{1}{m}. \]
Letting $m \to \infty$, we get $b \geq k+2 = \frac{b}{a} +2$. This is equivalent to 
\[ b \geq \frac{2}{1- \frac{1}{a}} \geq \frac{x}{1 - \frac{1}{a}}, \]
thus we get the desired conclusion. 

\medskip

\noindent In the second case we suppose there exists some $i \in \{1, …, d+1\}$ such that $\gamma_{(-k_i, -l_i)}$ has $l_i <0$. Assume that the corresponding component $F_i$ has $m_i$-many positive ends, then we get
\begin{equation}\label{area-Fi}
{\rm Area}(F_i) = m_i b + k_i + l_i x.
\end{equation}
The Fredholm index formula (\ref{ind-Fi}) and Lemma \ref{lemma-index} imply that $m_i(k+2) + k_i + l_i =1$. Note that $m_i \neq 0$, otherwise we have 
\[ {\rm Area}(F_i) = k_i + l_i x = 1- l_i + l_i x = 1 + l_i(x-1) <1 \]
since $l_i<0$ and $x>1$ by our hypothesis. In other words, we have a finite energy plane $F_i$ with area strictly less than 1, which contradicts Lemma 3.7 in \cite{HO19}. Substituting the Fredholm index relation into (\ref{area-Fi}), we get 
\[ {\rm Area}(F_i) = m_i (b-(k+2)) + (1-l_i + l_i x) \geq 0.\]
This simplifies to $1 + l_i(x-1) \geq m_i((k+2)-b)$. Since $l_i<0$, that is $l_i \leq -1$, we get $2-x \geq 1 -l_i (x-1)$. Hence, since $m_i \geq 1$, either $$b > k+2 = \frac{b}{a}+2$$ which gives $b > \frac{2}{1 - \frac{1}{a}}$ or else
\[ 2- x \geq m_i((k+2)-b) \geq \frac{b}{a} + 2 - b, \]
and we get $b \geq \frac{x}{1 - \frac{1}{a}}$, which is the desired conclusion. 

\medskip

Finally we discuss the proof of Theorem \ref{thm-gluing2}. By Example \ref{ex-1} the moduli space is nonempty at least when $m=1$, that is, cylinders exist. The inductive gluing procedure described in subsection \ref{indgluing} still applies. Hence, it suffices to show that $\mathcal N^s$ is nonempty when $m=2$. In other words, we need to show the existence of a simple $J$-holomorphic curve $u$ in the symplectic cobordism resulting from the embedding $E(\ep, \ep S) \hookrightarrow E(1,k)$ for $k \in \N_{\geq 2}$ with two positive ends $\alpha_2$ (simply covered long orbit of $E(1,k)$) and one negative end on $\beta_1^{2k+3}$ (multiply covered short orbit of $E(\ep, \ep S)$).

Implicit in Theorem 1.3.1 of \cite{Sie20} is that a sufficient condition for the existence of such a curve is  given by
\begin{equation} \label{algebra2}
\left<(\Phi^2_{\ep, \ep S} \circ (\Psi^1_{1,k} \odot \Psi^1_{1,k}))(A_{k+1} \odot A_{k+1}), A_{2k+3} \right> \neq 0.
\end{equation}

Here $\Psi_{1,k} = \{\Psi^n_{1,k}\}_{n \geq 1}$ is an $\mathcal L_{\infty}$ homomorphism from $V_{1,k}$ to $V$. We have that $V_{1,k}$ is a polynomial algebra with generators $\{A_i\}_{i \ge 1}$, where $A_i$ corresponds to the unique closed Reeb orbit on $E(1,k+ \ep)$ with Conley-Zehnder index $2i+1$. The algebra $V$ is described in \cite{Sie20}, subsection 2.2. It has generators $\alpha_{i,j}$ for $i,j \ge 1$ and $\beta_{i,j}$ for $i,j \ge 0$ and not both $0$. The $\mathcal L_{\infty}$ homomorphism $\Phi_{\ep, \ep S} = \{\Phi^n_{\ep, \ep S}\}_{n \geq 1}$ maps from $V$ to $V_{\ep, \ep S}$.  The algebra $V_{\ep, \ep S}$ can again by identified with the polynomial algebra generated by the $A_i$, but now we think of the $A_i$ as corresponding to Reeb orbits on $E(\ep, \ep S)$.
The notation ``$\odot$'' means the symmetric tensor product.

Our morphisms can be defined inductively as follows. 
\begin{itemize}
\item[(i)] $\Psi_{1,k}^1(A_q) = C_{q; 1,k} \beta_{i(q), j(q)}$, for some $C_{q; 1,k} \neq 0$.
\item[(ii)] $\Phi_{\ep, \ep S}^{1}(\beta_{i,j}) = \frac{i(q)! j(q)!}{i! j! C_{q; \ep, \ep S}} A_q$ with $q = i+j$ and $C_{q; \ep, \ep S} \neq 0$. 
\item[(iii)] $\Phi_{\ep, \ep S}^2(\beta_{i(q_1), j(q_1)}, \beta_{i(q_2), j(q_2)}) = 0$ for $q_1, q_2 \in \Z_{\geq 1}$.
\item[(iv)] If $(i_1,j_1), (i_2, j_2) \in \Z^2_{\geq 0} \backslash \{(0,0)\}$, then we have the following recursive formula, 
\[ j_1 \Phi^2_{\ep, \ep S}(\beta_{i_1-1, j_1}, \beta_{i_2, j_2}) - i_1 \Phi^2_{\ep, \ep S}(\beta_{i_1, j_1-1}, \beta_{i_2, j_2}) + (i_i j_2 - j_1 i_2) \Phi^1_{\ep, \ep S}(\beta_{i_1+i_2, j_1 + j_2}) = 0. \]
\end{itemize}

\smallskip

The optimal index $(i(q), j(q))$ is defined by the pair $(i,j)$ that realizes the minimum $\min_{i+j=q} \max\{ia, jb\}$ for $(a,b) = (1,k)$ or $(\ep, \ep S)$. Observe that 
\begin{equation} \label{q-th}
\min_{i+j=q} \max\{ia, jb\} = \,\left\{\begin{array}{l} \mbox{$q$-th smallest entry of the infinite} \\ \mbox{sequence $\{ic \, | i \in \Z_{\geq 1}, \, c \in \{a,b\}\}$}.\end{array} \right\}. 
\end{equation} 

\begin{ex} \label{ex-opt-pair} (a) For $(a,b) = (1,k)$ and $q = k+1$, we have $(i(q), j(q)) = (k,1)$. In fact, we have the infinite sequence
\[  \{ic \, | i \in \Z_{\geq 1}, \, c \in \{a,b\}\} = \{1, 2, \cdots, k-1, k, k, k+1,…\} \]
where the first $k = k  \cdot 1$ and the second $k = 1 \cdot k$. Therefore, the $(k+1)$-th smallest entry is $k$. Up to (\ref{q-th}), the only way to obtain $k$ for $\max\{i, jk\}$ is when $i = k$ and $j=1$. Thus we get the desired optimal index. 

\medskip

(b) For $(a,b) = (1,S)$ and $q = 2k+3$ with $S> 2k+3$, we have $(i(q), j(q)) = (2k+3,0)$. In fact, we have the infinite sequence
\[  \{ic \, | i \in \Z_{\geq 1}, \, c \in \{a,b\}\} = \{1, 2, \cdots, 2k+2, 2k+3, …, S, … \}.\]
Therefore, the $(2k+3)$-th smallest entry is $2k+3$. Up to (\ref{q-th}), the only way to obtain $2k+3$ for $\max\{i, jS\}$ is when $i = 2k+3$ and $j=0$. Thus we get the desired optimal index. The same conclusion holds for $(a,b) = (\ep, \ep S)$. 
\end{ex}

We can now confirm that
the algebraic condition (\ref{algebra2}) holds. 

\begin{prop} [cf.~Example 3.2.4 in \cite{Sie20}] \label{non-zero-coefficient} 
\[ \left<(\Phi^2_{\ep, \ep S} \circ (\Psi^1_{1,k} \odot \Psi^1_{1,k}))(A_{k+1} \odot A_{k+1}), A_{2k+3} \right> = \frac{C^2_{k+1; 1,k}}{C_{2k+3; \ep, \ep S}} \cdot (2k+3)(k^2+k), \]
in particular, the coefficient is non-zero. Hence, the algebraic condition (\ref{algebra2}) holds.
\end{prop}

\begin{proof} By the item (i) above and (a) in Example \ref{ex-opt-pair}
we have 
\[ \Psi^1_{1,k}(A_{k+1})= C_{q; 1,k} \beta_{i(k+1), j(k+1)} = C_{k+1; 1,k} \beta_{k,1}. \]
Next, we aim to compute $\Phi^2_{\ep, \ep S}(\beta_{k,1}, \beta_{k,1})$ (and leave the consideration of the non-zero constant $C_{k+1; 1,k}$ at the end). By the item (iv) above, we have 
\begin{equation} \label{recur-1}
\Phi^2_{\ep, \ep S}(\beta_{k,1}, \beta_{k,1}) = (k+1) \Phi^2_{\ep, \ep S}(\beta_{k+1,0}, \beta_{k,1}) - \Phi^1_{\ep,\ep S}(\beta_{2k+1,2}). 
\end{equation} 
Since $\Phi^2_{\ep, \ep S}(\beta_{k+1,0}, \beta_{k,1}) = \Phi^2_{\ep, \ep S}(\beta_{k,1}, \beta_{k+1, 0})$ by its defining property, by the item (iv) again, we have 
\begin{equation} \label{recur-2}
\Phi^2_{\ep, \ep S}(\beta_{k+1,0}, \beta_{k,1}) = (k+1) \Phi^2_{\ep, \ep S}(\beta_{k+1, 0}, \beta_{k+1, 0}) + (k+1) \Phi^1_{\ep, \ep S}(\beta_{2k+2, 1}). 
\end{equation}
Then (\ref{recur-1}) and (\ref{recur-2}) yield 
\begin{align*}
\Phi^2_{\ep, \ep S}(\beta_{k,1}, \beta_{k,1})  & = (k+1)^2 \Phi^2_{\ep, \ep S}(\beta_{k+1, 0}, \beta_{k+1, 0}) + (k+1)^2 \Phi^1_{\ep, \ep S}(\beta_{2k+2,1}) - \Phi^1_{\ep, \ep S}(\beta_{2k+1, 2}) \\
& = (k+1)^2 \Phi^1_{\ep, \ep S}(\beta_{2k+2,1}) - \Phi^1_{\ep, \ep S}(\beta_{2k+1, 2})\\
& = \frac{(k+1)^2(2k+3)! 0!}{(2k+2)!1! C_{{2k+3}; \ep, \ep S}} A_{2k+3} - \frac{(2k+3)! 0!}{(2k+1)!2! C_{{2k+3}; \ep, \ep S}} A_{2k+3}\\
& = \left(\frac{(2k+3)((k+1)^2 - (k+1))}{C_{2k+3; \ep, \ep S}}\right) A_{2k+3}\\
& = \frac{(2k+3)(k^2+k)}{C_{2k+3; \ep, \ep S}}A_{2k+3}.
\end{align*}
The second equality is due to the item (iii) above, and the third equality is due to the item (ii) above together with (b) in Example \ref{ex-opt-pair}. Finally, since $\Phi^2_{\ep, \ep S}$ is linear, a consideration of the constant $C_{k+1; 1,k}$ leads to the required conclusion. 
\end{proof}

\begin{remark} The constant $\frac{C^2_{k+1; 1,k}}{C_{2k+3; \ep, \ep S}}$ in the conclusion of Proposition \ref{non-zero-coefficient} can be computed explicitly; see the hypothesis of Theorem 5.3.2 in \cite{Sie20}. \end{remark}

\bibliographystyle{amsplain}
\bibliography{biblio_hsi}

\providecommand{\bysame}{\leavevmode\hbox to3em{\hrulefill}\thinspace}
\providecommand{\MR}{\relax\ifhmode\unskip\space\fi MR }
\providecommand{\MRhref}[2]{%
  \href{http://www.ams.org/mathscinet-getitem?mr=#1}{#2}
}
\providecommand{\href}[2]{#2}
\begin{thebibliography}{10}

\bibitem{Cas14}
Casim Abbas, \emph{An introduction to compactness results in symplectic field
  theory}, Springer, Heidelberg, 2014. \MR{3157146}

\bibitem{BEHWZ03}
Fr{\'e}d{\'e}ric Bourgeois, Yakov Eliashberg, Helmut Hofer, Krzysztof Wysocki,
  and Eduard Zehnder, \emph{Compactness results in symplectic field theory},
  Geom. Topol. \textbf{7} (2003), 799--888. \MR{2026549}

\bibitem{Che96}
Yu.~V. Chekanov, \emph{Hofer's symplectic energy and {L}agrangian
  intersections}, Publ. Newton Inst., vol.~8, Cambridge Univ. Press, Cambridge,
  1996. \MR{1432467}

\bibitem{chesch}
Yuri Chekanov and Felix Schlenk, \emph{Lagrangian product tori in symplectic
  manifolds}, Comment. Math. Helv. \textbf{91} (2016), no.~3, 445--475.
  \MR{3541716}

\bibitem{Choi16}
Keon Choi, \emph{Combinatorial embedded contact homology for toric contact
  manifolds}, arXiv: 1608.07988.

\bibitem{CM18}
K.~Cieliebak and K.~Mohnke, \emph{Punctured holomorphic curves and {L}agrangian
  embeddings}, Invent. Math. \textbf{212} (2018), no.~1, 213--295. \MR{3773793}

\bibitem{C-GH18}
Daniel Cristofaro-Gardiner and Richard Hind, \emph{Symplectic embeddings of
  products}, Comment. Math. Helv. \textbf{93} (2018), no.~1, 1--32.
  \MR{3777123}

\bibitem{DNNWY20}
Leo Digiosia, Jo~Nelson, Haoming Ning, Morgan Weiler, and Yirong Yang,
  \emph{Symplectic embeddings of four-dimensional polydisks into half integer
  ellipsoids}, arXiv: 2010.06687.

\bibitem{D-RGI16}
Georgios Dimitroglou~Rizell, Elizabeth Goodman, and Alexander Ivrii,
  \emph{Lagrangian isotopy of tori in {$S^2\times S^2$} and {$\Bbb{C}P^2$}},
  Geom. Funct. Anal. \textbf{26} (2016), no.~5, 1297--1358. \MR{3568033}

\bibitem{Eli91}
Yakov Eliashberg, \emph{New invariants of open symplectic and contact
  manifolds}, J. Amer. Math. Soc. \textbf{4} (1991), no.~3, 513--520.
  \MR{1102580}

\bibitem{EGH00}
Yakov Eliashberg, Alexander Givental, and Helmut Hofer, \emph{Introduction to
  symplectic field theory}, Geom. Funct. Anal. (2000), no.~Special Volume, Part
  II, 560--673, GAFA 2000 (Tel Aviv, 1999). \MR{1826267}

\bibitem{GH18}
Jean Gutt and Michael Hutchings, \emph{Symplectic capacities from positive
  {$S^1$}-equivariant symplectic homology}, Algebr. Geom. Topol. \textbf{18}
  (2018), no.~6, 3537--3600. \MR{3868228}

\bibitem{HK14}
Richard Hind and Ely Kerman, \emph{New obstructions to symplectic embeddings},
  Invent. Math. \textbf{196} (2014), no.~2, 383--452. \MR{3193752}

\bibitem{HK18}
\bysame, \emph{Correction to: {N}ew obstructions to symplectic embeddings},
  Invent. Math. \textbf{214} (2018), no.~2, 1023--1029. \MR{3867635}

\bibitem{HK18-2}
\bysame, \emph{{$J$}-holomorphic cylinders between ellipsoids in dimension
  four}, J. Symplectic Geom. \textbf{18} (2020), no.~5, 1221--1245.
  \MR{4174300}

\bibitem{HO19}
Richard Hind and Emmanuel Opshtein, \emph{Squeezing {L}agrangian tori in
  dimension 4}, Comment. Math. Helv. \textbf{95} (2020), no.~3, 535--567.
  \MR{4152624}

\bibitem{HZip}
Richard Hind and Jun Zhang, \emph{Shape invariant and symplectic
  {B}anach-{M}azur distances}, In progress.

\bibitem{Hof06}
Helmut Hofer, \emph{A general {F}redholm theory and applications}, Int. Press,
  Somerville, MA, 2006. \MR{2459290}

\bibitem{Hut11}
Michael Hutchings, \emph{Quantitative embedded contact homology}, J.
  Differential Geom. \textbf{88} (2011), no.~2, 231--266. \MR{2838266}

\bibitem{Hut14}
\bysame, \emph{Lecture notes on embedded contact homology}, Bolyai Soc. Math.
  Stud., vol.~26, J\'{a}nos Bolyai Math. Soc., Budapest, 2014. \MR{3220947}

\bibitem{Hut16}
\bysame, \emph{Beyond {ECH} capacities}, Geom. Topol. \textbf{20} (2016),
  no.~2, 1085--1126. \MR{3493100}

\bibitem{HT07}
Michael Hutchings and Clifford~Henry Taubes, \emph{Gluing pseudoholomorphic
  curves along branched covered cylinders. {I}}, J. Symplectic Geom. \textbf{5}
  (2007), no.~1, 43--137. \MR{2371184}

\bibitem{HT13}
\bysame, \emph{Proof of the {A}rnold chord conjecture in three dimensions,
  {II}}, Geom. Topol. \textbf{17} (2013), no.~5, 2601--2688. \MR{3190296}

\bibitem{Irv19}
Daniel Irvine, \emph{The stabilized symplectic embedding problem for
  polydiscs}, arXiv: 1907.13159.

\bibitem{McD18}
Dusa McDuff, \emph{A remark on the stabilized symplectic embedding problem for
  ellipsoids}, Eur. J. Math. \textbf{4} (2018), no.~1, 356--371. \MR{3782228}

\bibitem{Pol91}
Leonid Polterovich, \emph{The {M}aslov class of the {L}agrange surfaces and
  {G}romov's pseudo-holomorphic curves}, Trans. Amer. Math. Soc. \textbf{325}
  (1991), no.~1, 241--248. \MR{992608}

\bibitem{RS93}
Joel Robbin and Dietmar Salamon, \emph{The {M}aslov index for paths}, Topology
  \textbf{32} (1993), no.~4, 827--844. \MR{1241874}

\bibitem{Sch-book}
Felix Schlenk, \emph{Embedding problems in symplectic geometry}, De Gruyter
  Expositions in Mathematics, vol.~40, Walter de Gruyter GmbH \& Co. KG,
  Berlin, 2005. \MR{2147307}

\bibitem{Sie20}
Kyler Siegel, \emph{Computing higher symplectic capacities {I}}, arXiv:
  1911.06466.

\bibitem{Sieip}
\bysame, \emph{Computing higher symplectic capacities {I}{I}}, In progress.

\bibitem{Sik89}
Jean-Claude Sikorav, \emph{Rigidit\'{e} symplectique dans le cotangent de
  {$T^n$}}, Duke Math. J. \textbf{59} (1989), no.~3, 759--763. \MR{1046748}

\bibitem{Sik91}
\bysame, \emph{Quelques propri\'{e}t\'{e}s des plongements lagrangiens},
  M\'{e}m. Soc. Math. France (N.S.) (1991), no.~46, 151--167, Analyse globale
  et physique math\'{e}matique (Lyon, 1989). \MR{1125841}

\bibitem{SZ18}
Vuka\v{s}in Stojisavljevi\'{c} and Jun Zhang, \emph{Persistence modules,
  symplectic {B}anach-{M}azur distance and {R}iemannian metrics},
  arXiv:1810.11151.

\bibitem{Ush19}
Michael Usher, \emph{Symplectic {B}anach-{M}azur distances between subsets of
  {$ \C ^n$}}, J. Topol. Anal. (2019).

\bibitem{Vit90}
Claude Viterbo, \emph{A new obstruction to embedding {L}agrangian tori},
  Invent. Math. \textbf{100} (1990), no.~2, 301--320. \MR{1047136}

\bibitem{wendlauto}
C.~Wendl, \emph{Automatic transversality and orbifolds of punctured holomorphic
  curves in dimension four.}, Comment. Math. Helv. \textbf{85} (2010), no.~2,
  347--407.

\bibitem{Wen16}
Chris Wendl, \emph{Lectures on symplectic field theory}, To appear in {E}{M}{S}
  {L}ectures in {M}athematics series, arXiv: 1612.01009.

\end{thebibliography}
\noindent\\
\end{document}